\documentclass[10pt]{amsart}

\usepackage{times}
\usepackage{amsmath,amssymb}
\usepackage{hyperref}
\usepackage{graphicx}

\usepackage{geometry}
\newtheorem{theorem}{Theorem}
\newtheorem{lemma}{Lemma}

\newtheorem{definition}{Definition}

\newtheorem{conjecture}{Conjecture}

\begin{document}

\title{On the inverse power index problem}

\author{Sascha Kurz}
\address{Sascha Kurz\\Department for Mathematics, Physics and Informatics\\University Bayreuth\\Germany}
\email{sascha.kurz@uni-bayreuth.de}


\maketitle

\begin{abstract}
  Weighted voting games are frequently used in decision making. Each voter
  has a weight and a proposal is accepted if the weight sum of the
  supporting voters exceeds a quota. One line of research is the efficient
  computation of so-called power indices measuring the influence of a
  voter. We treat the inverse problem: Given an influence vector and a
  power index, determine a weighted voting game such that the distribution
  of influence among the voters is as close as possible to the given target
  value. We present exact algorithms and computational results for the
  Shapley-Shubik and the (normalized) Banzhaf power index.
\end{abstract}

\section{Introduction}

The European Economic Community, i.~e.{} the predecessor of the European Union, was founded in 1957 by
Belgium, France, Germany, Italy, Luxembourg and the Netherlands. Decisions were and are made via polling.
More precisely, each country has a given voting weight, see Table~\ref{table_voting_weights_EEC}\footnote{Recent
estimates for the population are taken from http://epp.eurostat.ec.europa.eu.}, and a proposal is approved if it
is supported by at least $12$, called the quota, out of the 17~votes. So Luxembourg possesses $\frac{1}{17}$ of all
votes. But let us consider the following thought experiment: The five countries except Luxembourg vote for their own
without asking for Luxembourg's opinion. In any case the number $s$ of supporting votes is even. Thus the vote of
Luxembourg does not matter at all. For $s\ge 12$ the proposal will be approved and for $s\le 10$ the proposal will
be rejected. We may say that Luxembourg has absolutely no power in this weighted voting game.

\begin{table}[htp]
  \begin{center}
    \begin{tabular}{rrr}
      \hline
      state           & votes & population (2010)\\
      \hline
      Germany         & 4 & 82144902 \\
      France          & 4 & 62582650 \\
      Italy           & 4 & 60017346 \\
      the Netherlands & 2 & 16503473 \\
      Belgium         & 2 & 10783738 \\
      Luxembourg      & 1 &   494153 \\
      \hline\\
    \end{tabular}
    \caption{Voting weights and the quota in the European Economic Community of 1957.}
    \label{table_voting_weights_EEC}
  \end{center}
\end{table}

To have a measurement of the relative power of the members of voting schemes in more complex situations, power
indices like the Shapley-Shubik or the Banzhaf power index were invented, see the next section for a definition.
Table~\ref{table_power_indices} lists the values of the two mentioned power indices for three different weighted
voting games. We remark that there is a non-linear dependency between weights, quota and the value of the considered
power index.

\begin{table}[htp]
  \begin{center}
    \begin{tabular}{rrrrrrrrrr}
      \hline
      state           & $w_i$ & $\mathcal{SS}$ & $\mathcal{BZ}$ & $w_i$ & $\mathcal{SS}$ & $\mathcal{BZ}$ &
      $w_i$ & $\mathcal{SS}$ & $\mathcal{BZ}$\\
      \hline
      Germany         &  4 & 0.233 & 0.238 & 4 & 0.233 & 0.229 & 4 & 0.233 & 0.222 \\[1mm]
      France          &  4 & 0.233 & 0.238 & 4 & 0.233 & 0.229 & 4 & 0.233 & 0.222 \\[1mm]
      Italy           &  4 & 0.233 & 0.238 & 4 & 0.233 & 0.229 & 4 & 0.233 & 0.222 \\[1mm]
      Netherlands     &  2 & 0.150 & 0.143 & 3 & 0.200 & 0.188 & 2 & 0.100 & 0.111 \\[1mm]
      Belgium         &  2 & 0.150 & 0.143 & 2 & 0.050 & 0.063 & 2 & 0.100 & 0.111 \\[1mm]
      Luxembourg      &  1 & 0.000 & 0.000 & 1 & 0.050 & 0.063 & 1 & 0.100 & 0.111 \\[1mm]
      \hline
      quota           & 12 &       &       &12 &       &       &11 &       & \\
      \hline\\
    \end{tabular}
    \caption{Power indices for different weighted voting games.}
    \label{table_power_indices}
  \end{center}
\end{table}

If we suppose that there exists some agreement on a fair distribution of power $\left(d_i\right)_{1\le i\le n}$
among the $n$~countries and a selection of an appropriate power index $\mathcal{P}$, we obtain the inverse power
index problem\footnote{In \cite{keijzer2} the authors call this problem the \textit{voting game design problem}.}:
$\min \Vert d-\mathcal{P}(\chi)\Vert$, where $\chi$ is a weighted voting game for $n$~voters and $\Vert\cdot\Vert$
an arbitrary suitable norm to measure the deviation, i.~e.{} a rule to decide which assignment of voting weights is
\textit{fairer}. In practice the intended power distribution $d$ may arise from the population numbers, either in the
variant ``One Person, One Vote'', i.~e.{} $d_i$ is proportional to the population of country $i$, or by using Penrose
square root law, i.~e.{} $d_i$ is proportional to the square root of the population of country~$i$ \cite{Penrose,Polen}.

\subsection{Related work}

There is a vast literature how to compute the Shapley-Shubik index and other power indices in various circumstances,
either exactly or approximatively. Few results are known about the inverse power index problem so far. Leech proposes
a certain kind of fixed point algorithm and reports that it works considerably well in practice, whenever the number of
voters $n$ is not too small \cite{Leech1,Leech2}, see also \cite{lastei}. The author argues that for large $n$ one may
assume that the power index smoothly depends on the voting weights and thus Brouwer's fixed point theorem can be applied
to \textit{prove} the convergence of this approach. Another heuristic is described in \cite{anytime}.

Besides using the finiteness of the set of weighted voting games, the first general bounds stating that some power
distributions can not be approximated by Banzhaf vectors too closely are given by Alon and Edelman \cite{pre05681536}.

Since for each number of voters $n$ there is only a finite set of weighted voting games one may simply solve the
inverse power index problem by looping over the whole set. The enumeration of weighted voting games dates back to
at least 1962 \cite{0105.12002}, where up to $6$~voters are treated. For $n=7,8$ voters we refer e.~g.{} to
\cite{0233.94016,0205.17805,0841.90134}. Bart de Keijzer presents a promising graded poset for weighted voting
games in his master thesis \cite{keijzer}, see also \cite{keijzer2}.\footnote{We would like to remark that the
counts for weighted voting games with $6\le n\le 8$ voters are wrongly stated in \cite{keijzer}, but the methods
should work.} $n=9$ voters  where successfully treated in \cite{minimum_representation,da_tautenhahn}.

A similar inverse problem was solved in \cite{0942.91501}.

\subsection{Our contribution}
We present the first exact algorithm, besides complete enumeration of weighted voting games, for the inverse power
index problem. We computationally show that for $n\le 15$ voters the Shapley-Shubik vector of every simple game, complete
simple game or weighted voting game has a $\Vert\cdot\Vert_1$-distance to $(0.75,0.25,0\dots)$ of at least $\frac{1}{3}$.
For the Banzhaf power index the (conjectured) bound for the worst possible approximation in $\Vert\cdot\Vert_1$-distance
is a bit more involved. But it seems that for $n\ge 2$ voters $(0.75,0.25,0\dots)$ is one desired power distribution
where the worst case is attained. We present computational results for $n\le 11$ for simple games and $n\le 16$ for
complete simple games and weighted voting games. For more realistic instances we consider subsets of the European Union
and determine the exact solutions for the inverse power index problem for small numbers of voters. Additionally we present
some results on the set of achievable power distributions for the Banzhaf and the Shapley Shubik power index for $n\le 9$ voters.

\subsection{Outline of the paper}
In Section~\ref{section_basics} we state the definitions of the three voting methods simple games, complete simple games
and weighted voting games. We further introduce the Banzhaf and the Shapley Shubik power index. The known a priori estimates
are the topic of Section~\ref{section_a_priori}. We give precise conjectures about the, with respect to approximation via Banzhaf and Shapley Shubik vectors, worst case power distributions. In Section~\ref{section_ilp} we give an exact algorithm for the inverse power index problem for simple games, complete simple games and weighted voting games for the Banzhaf and the Shapley Shubik power index via an ILP formulation. Another exact approach via exhaustive enumeration is described in Section~\ref{section_exhaustive_enumeration}. A promising combination of both methods as a branch\&bound approach is outlined in Subsection~\ref{subsec_b_and_b}. Computational results evaluating the proposed exact algorithms for the inverse power index problem and proving the stated conjectures for small numbers of voters are given in Section~\ref{section_computational_results}. To get a more global understanding we consider the sets of achievable Banzhaf and Shapley Shubik vectors in Section~\ref{sec_hard_to_approximate} and close with a conclusion in Section~\ref{sec_hard_to_approximate}. Additionally we classify the complete set of desired power distributions which have the worst case approximation property for very small~$n$ in an appendix.

\section{Voting methods and power indices}
\label{section_basics}

Consider a yes-no voting system for a set $N=\{1,2,\dots,n\}$ of $n$ voters. The acceptance of a proposal depends on the subset of the supporting voters. In general this can be described by a Boolean function $\chi:2^N\rightarrow\{0,1\}$. A quite natural requirement for a voting system is monotonicity, i.~e.{} $\chi(U)\le\chi(U')$ for all $U\subseteq U'\subseteq N$. Monotone Boolean functions $\chi:2^N\rightarrow\{0,1\}$ with $\chi(\emptyset)=0$ and $\chi(N)=1$ are called simple games.

\begin{definition}
  For a simple game $\chi:2^N\rightarrow\{0,1\}$ a subset $U\subseteq N$ is called a minimal winning coalition if
  $\chi(U)=1$ and $\chi(U\backslash\{i\})=0$ for all $i\in U$. Analogously a subset $U\subseteq N$ is a maximal losing
  coalition if $\chi(U)=0$ and $\chi(U\cup\{i\})=1$ for all $i\in N\backslash U$.
\end{definition}

Each simple game is completely characterized by either the complete list $\overline{W}$ of its minimal winning coalitions or the complete list $\overline{L}$ of its maximal losing coalitions. Both the set of minimal winning coalitions and the set of the maximal losing coalitions are antichains. For each simple game $\chi$ there is a dual simple game $\chi^d$ defined via $\chi^d(U)=1-\chi(N\backslash U)$.
A voter $i\in N$ is called a null voter if he is not a member of at least one minimal winning coalition.

In many applications voting systems are restricted to smaller sub classes of simple games, e.~g.{} complete simple games arising from Isbell's desirability relation \cite{0083.14301}: We write $i\sqsupset j$ for two voters $i,j\in N$ iff we have $\chi\Big(\{i\}\cup U\backslash\{j\}\Big)\ge\chi (U)$ for all $j\in U\subseteq N\backslash\{i\}$. A pair $(N,\chi)$ is called complete simple game if it is a simple game and the binary relation $\sqsupset$ is a total preorder and we abbreviate $i\sqsupset j$, $j\sqsupset i$ by $i~\square~j$. To factor out symmetries we assume $1\sqsupset 2 \sqsupset \dots \sqsupset n$ and write coalitions $U\subseteq N$ as characteristic vectors $u$. We call $u$ a winning vector iff $\chi(U)=1$, otherwise we call $u$ a losing vector. 

\begin{figure}[htp]
  \begin{center}
    \setlength{\unitlength}{1cm}
    \begin{picture}(9,2.2)
      \put(0,1){$\begin{pmatrix}0\\0\\0\end{pmatrix}$}
      \put(1.4,1){$\begin{pmatrix}0\\0\\1\end{pmatrix}$}
      \put(2.8,1){$\begin{pmatrix}0\\1\\0\end{pmatrix}$}
      \put(4.2,1.8){$\begin{pmatrix}1\\0\\0\end{pmatrix}$}
      \put(4.2,0.2){$\begin{pmatrix}0\\1\\1\end{pmatrix}$}
      \put(5.6,1){$\begin{pmatrix}1\\0\\1\end{pmatrix}$}
      \put(7.0,1){$\begin{pmatrix}1\\1\\0\end{pmatrix}$}
      \put(8.4,1){$\begin{pmatrix}1\\1\\1\end{pmatrix}$}
      \put(0.9,1.1){\vector(1,0){0.45}}
      \put(2.3,1.1){\vector(1,0){0.45}}
      \put(3.7,1.2){\vector(1,1){0.5}}
      \put(3.7,1.0){\vector(1,-1){0.5}}
      \put(5.1,1.70){\vector(1,-1){0.5}}
      \put(5.1,0.50){\vector(1,+1){0.5}}
      \put(6.5,1.1){\vector(1,0){0.45}}
      \put(7.9,1.1){\vector(1,0){0.45}}
    \end{picture}
    \caption{The Hasse diagram for $\preceq$ on $N=\{1,2,3\}$.}
    \label{fig_hasse}
  \end{center}
\end{figure}
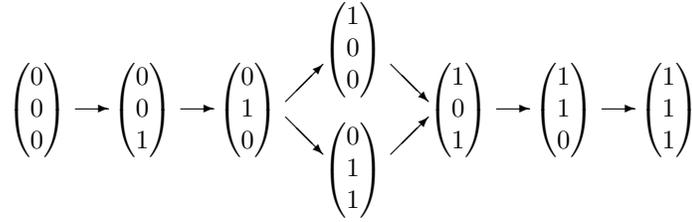

To state an analogue of minimal winning coalitions for complete simple games we need another partial ordering: For two coalitions $u=\left(\begin{array}{ccc}u_1&\dots&u_n\end{array}\right), v=\left(\begin{array}{ccc}v_1&\dots&v_n\end{array}\right)\in\{0,1\}^n$ we write $u\preceq v$ iff we have $\sum\limits_{i=1}^{k} u_i \le \sum\limits_{i=1}^{k} v_i$ for all $1\le k\le n$ and abbreviate $u\preceq v$, $u\neq v$ by $u\prec v$, $u \npreceq v$, $u \nsucceq v$ by $u\bowtie v$. With this we have $\chi(u)\le \chi(v)$ for all $u\preceq v$. By $\widehat{W}$ we denote the set of all shift-minimal winning vectors, i.~e.{} winning vectors which are minimal with respect to $\preceq$. Similarly we denote by $\widehat{L}$ the set of all shift-maximal losing vectors, i.~e.{} losing vectors which are maximal with respect to $\preceq$. We remark that each complete simple game is uniquely characterized by either $\widehat{W}$ or $\widehat{L}$, see e.~g.{} \cite{complete_simple_games}. In Figure~\ref{fig_hasse} we depict the Hasse diagram for $\preceq$ on $N=\{1,2,3\}$ and see that there are exactly $8$~complete simple games (or antichains) for $n=3$ voters; $7$ with one shift-minimal winning vector and one with two shift-minimal winning vectors.

An even more important sub class of complete simple games are weighted voting games given by non-negatives weights $w_i$ for all $i\in N$ and a quota $q$. The corresponding monotone Boolean function $\chi$ is obtained by $$\chi(U)=1\quad\quad\Longleftrightarrow\quad\quad \sum\limits_{i\in U}w_i\ge q.$$
Algorithmically we can check whether a given simple game~$\chi$ is a weighted voting game by solving a suitable linear program, see Section~\ref{section_exhaustive_enumeration}. 

\begin{table}[htp]
\begin{center}
\begin{tabular}{rrrrrrrrrr}
\hline
$\!\!\!\mathbf{n}\!\!\!$&$\!\!1\!\!$&$\!\!2\!\!$&$\!\!3\!\!$&$\!\!4\!\!$&$\!\!5\!\!$&$\!\!6\!\!$&
$\!\!\!7\!\!\!$&$\!\!\!8\!\!\!$&$\!\!\!9\!\!\!$\\
$\!\!\!\mathbf{\#}\text{\textbf{sg}}\!\!\!$&\!\!1\!\!&\!\!4\!\!&\!\!18\!\!&\!\!166\!\!&\!\!7579\!\!&\!\!7828352\!\!&\!\!2414682040996\!\!&\!\!$5.613043\cdot 10^{22}$\!\!
&\!\!$>10^{42}$\!\!\\
$\!\!\!\mathbf{\#}\text{\textbf{csg}}\!\!\!$&\!\!1\!\!&\!\!3\!\!&\!\!8\!\!&\!\!25\!\!&\!\!117\!\!&\!\!1171\!\!&
\!\!44313\!\!&\!\!16175188\!\!&\!\!284432730174\!\!\\
$\!\!\!\mathbf{\#}\text{\textbf{wvg{}}}\!\!\!$&\!\!1\!\!&\!\!3\!\!&\!\!8\!\!&\!\!25\!\!&\!\!117\!\!&\!\!1111\!\!&\!\!29373\!\!&\!\!2730164\!\!&\!\!989913344\!\!\\
\hline\\
\end{tabular}
\caption{Number of simple games, complete simple games and weighted voting games for $n$ voters.}
\label{table_numbers}
\end{center}
\end{table}

In Table~\ref{table_numbers} we state the known numbers of simple games, complete simple games and weighted voting games for $n$ voters, see \cite{csg,minimum_representation,da_tautenhahn,0736.06017}. To define two of the most important power indices for these three types of voting methods we need another definition:

\begin{definition}
  For a voter $i\in N$ and a monotone Boolean function $\chi$, an $i$-swing is a set $U\subseteq N\backslash\{i\}$ such
  that $\chi(U)=0$ and $\chi(U\cup\{i\})=1$.
\end{definition}

With this we can express the Banzhaf vector $(\mathcal{BZ}(\chi,1),\dots,\mathcal{BZ}(\chi,n))$ as $$\mathcal{BZ}(\chi,i)=\frac{1}{m}\cdot\left|\{U\subseteq N\backslash\{i\}\mid U\text{ is an $i$-swing}\}\right|,$$ where $m$ is the total number of swings. Also the Shapley-Shubik vector $(\mathcal{SS}(\chi,1),\dots,\mathcal{SS}(\chi,n))$ can be expressed using swings via $$\mathcal{SS}(\chi,i)=\frac{1}{n!}\cdot \sum\limits_{U\text{ is an $i$-swing }}
|U|!\cdot (n-|U|-1)!.$$

Since a set $U$ is an $i$-swing in a simple game $\chi$ if and only if it is an $i$-swing in the dual game $\chi^d$, both the Shapley-Shubik vector and the Banzhaf vector coincide for pairs of dual simple games. For both power indices the value for a null voter is zero. Removing a null voter from a simple game with $n$ voters gives a simple game with $n-1$ voters, where the power indices of the remaining $n-1$ voters coincide.

\section{A priori estimates}
\label{section_a_priori}

Since heuristics can only find good solutions but are not capable of providing lower bounds we are considering a priori estimates. The first fact coming to mind is the finite number of weighted voting games with $n$ voters. An upper bound of $2^{n^2-n+2}$ is given in \cite{keijzer2}. Respecting symmetry we conclude that for $n\ge 2$ there are at most $2^{n^2-n+2}\cdot n!<2^{n^2}\cdot n^n$ different Banzhaf- or Shapley-Shubik vectors in $[0,1]^n$. Decomposing $[0,1]^n$ into $2^{n^2}\cdot n^n$ subcubes of side lenght $\frac{1}{2^{n}n}$ results in at least one empty subcube using the pigeonhole principle. Thus there are points in $[0,1]^n$ such that the nearest Banzhaf- or Shapley-Shubik vector is at least $\frac{1}{2^{n+1}n}$ apart.

Having a closer look at the swings one can obtain a slightly sharper bound. For $n$ voters the number of swings lies between $n$ and $\left(\left\lfloor\frac{n}{2}\right\rfloor+1\right)\cdot{n\choose \left\lfloor\frac{n}{2}\right\rfloor+1}$, see \cite{0954.91019}. Using the fact that the numbers $\eta_i$ of swings for voter~$i$ all have the same parity we can conclude that either $\mathcal{BZ}(\chi,i)=\mathcal{BZ}(\chi,j)$ or we have
$$
  \left|\mathcal{BZ}(\chi,i)-\mathcal{BZ}(\chi,j)\right|\ge \frac{2}{\left(\left\lfloor\frac{n}{2}\right\rfloor+1\right)
  \cdot{n\choose \left\lfloor\frac{n}{2}\right\rfloor+1}}
$$
for voters $i$ and $j$.

Both bounds tend to zero when $n$ tends to infinity. Recently Alon and Edelman \cite{pre05681536} found a lower bound that does not tend to zero:
\begin{theorem}
  \label{thm_alon}
  Let $n>k$ be positive integers, let $\varepsilon<\frac{1}{k+1}$ be a positive real, and let $\chi$ be
  a simple game for $n$ voters. If $\sum\limits_{i=k+1}^nB(\chi,i)\le\varepsilon$, then there exists a simple game $\chi'$ for
  $k$ voters such that
  $$
    \Vert B(\chi)-B(\chi')\Vert_1=\sum_{i=1}^k \left|B(\chi,i)-B(\chi',i)\right|\,+\,\sum_{i=k+1}^n \left|B(\chi,i)-0\right|
    \le\frac{(2k+1)\varepsilon}{1-(k+1)\varepsilon}+\varepsilon.
  $$
\end{theorem}

Alon and Edelman have chosen $k=2$ and $\varepsilon=0.01$ to deduce that for a simple game with $\sum\limits_{i=3}^n B(\chi,i)\le 0.01$ the two-dimensional vector $(B(\chi,1),B(\chi,2))$ lies within $\Vert\cdot\Vert_1$-distance smaller than $\frac{1}{16}$ of one of the achievable Banzhaf vectors, i.~e.{} $(1,0)$, $(0,1)$, or $\left(\frac{1}{2},\frac{1}{2}\right)$. Thus for $n\ge 2$ voters the Banzhaf vector of every simple game has an $\Vert\cdot\Vert_1$-distance of at least $\min\left(0.5-\frac{1}{16},0.02\right)=0.02$ to the power distribution $(0.75,0.25,0,\dots)$.

We may improve this argument a bit by choosing $k=2$ and $\varepsilon=\frac{1}{18}$ to deduce that the minimal $\Vert\cdot\Vert_1$-distance to the power distribution $(0.75,0.25,0,\dots)$ is at least $\frac{1}{9}$, which is best possible using the general theorem.

\begin{conjecture}
  \label{conjecture_1}
  For $n\ge 2$ voters consider the following desired power distribution: $d_1=0.75$, $d_2=0.25$, and $d_i=0$ for $3\le i\le n$. For each simple game (complete simple game or
  weighted voting game) $\chi$ we have the inequalities 
  $
    \Vert \mathcal{SS}(\chi)-d\Vert_1=\sum\limits_{i=1}^n\left|\mathcal{SS}(\chi,i)-d_i\right|\ge \frac{1}{3}
  $
  and
  $
    \Vert \mathcal{BZ}(\chi)-d\Vert_1=\sum\limits_{i=1}^n\left|\mathcal{BZ}(\chi,i)-d_i\right|\ge
    \frac{k_{\left\lceil\frac{n}{2}\right\rceil}}{l_{\left\lceil\frac{n}{2}\right\rceil}},
  $
  where
  $
  k_1=1,\,\, k_m=\left\{\begin{array}{rcl}2k_{m-1}&:&m\equiv0\mod 2,\\8k_{m-1}-1&:&m\equiv 1\mod 2\end{array}\right.
  $
  for $m\ge 2$ and
  $
  l_1=2,\, \, l_2=5,\,\, l_m=\left\{\begin{array}{rcl}2l_{m-1}+3&:&m\equiv0\mod 2,\\8l_{m-1}-2&:&m\equiv 1\mod 2\end{array}\right.\text{ for }m\ge 3.
  $
\end{conjecture}

\begin{table}[htp]
  \begin{center}
    \begin{tabular}{rcccccc}
    \hline
     $\mathbf{m}$                & 2                 & 3                 & 4                              & 5                              & 6 & 7\\
     $\mathbf{\frac{k_m}{l_m}}$  &$\!\!\frac{1}{2}=0.5\!\!$&$\!\!\frac{2}{5}=0.4\!\!$&$\!\!\frac{15}{18}\approx 0.39474\!\!$&$\!\!\frac{30}{79}\approx 0.37975\!\!$&
     $\!\!\frac{239}{630}\approx 0.37937\!\!$&$\!\!\frac{478}{1263}\approx 0.37846\!\!$\\
     \hline 
    \end{tabular}
    \caption{Values of $\frac{k_m}{l_m}$ -- the conjectured lower bound for the Banzhaf index.}
    \label{table_conjectured_lower_bounds_banzhaf}
  \end{center}
\end{table}

Using $\tau(n)=\frac{(-1)^{n+1}+1}{2}$ a closed expression for these numbers $\frac{k_m}{l_m}$, see also Table~\ref{table_conjectured_lower_bounds_banzhaf}, is given by
$k_m=\frac{7\cdot 2^{2m-2+\tau(m)}+2-\tau(m)}{15}$ and $l_m=2^{2m-3+\tau(m)}+\frac{11\cdot2^{2m-2+\tau(m)}+1-23\tau(m)}{15}$ for $m\ge 2$. Thus we have $\lim\limits_{m\to\infty}\frac{k_m}{l_m}=\frac{14}{37}\approx 0.378378>\frac{1}{3}$.

Conjecture~\ref{conjecture_1} looks rather specific. A first hint why $d_1=0.75$, $d_2=0.25$ might be hard to approximate is the fact that this point maximizes the $\Vert\cdot\Vert_1$-distance to the achievable power vectors $(1,0)$, $(0.5,0.5)$, $(0,1)$, both for the Shapley-Shubik and the Banzhaf index. For some small $n>2$ we compute the points or regions which maximize the $\Vert\cdot\Vert_1$-distance to the achievable power vectors in Appendix~\ref{appendix_a}. So far we have not seen a geometric structure behind the values $\frac{k_m}{l_m}$ for the Banzhaf index, but computing these numbers for simple games and weighted voting games with $n\le 16$ suggests that the rather structured recursion formulas for $k_m$ and $l_m$ might be more than a pure coincidence. Having Theorem~\ref{thm_alon} in mind one might conjecture that power distributions with many zeros are hard to approximate whenever they are hard to approximate without the zero values. A first step towards understanding the inverse power index problem might be the results on achievable hierarchies in simple games, see \cite{achievable}.

Additionally we conjecture that there is a similar result as Theorem~\ref{thm_alon} for the Shapley-Shubik index. In the next sections we will present some exact algorithms to solve the inverse power index problem and use them to prove Conjecture~\ref{conjecture_1} for small $n$.

\begin{conjecture}
  \label{conjecture_2}
  For each desired power distribution $d_i$ with $d_i\ge 0$ and $\sum\limits_{i=1}^nd_i=1$, where $n\ge 3$, there exist two weighted voting games
  (complete simple games or simple games) $\chi$, $\chi'$ with
  $
    \Vert \mathcal{BZ}(\chi)-d\Vert_1=\sum\limits_{i=1}^n\left|\mathcal{BZ}(\chi,i)-d_i\right|\le
    \frac{k_{\left\lceil\frac{n}{2}\right\rceil}}{l_{\left\lceil\frac{n}{2}\right\rceil}}
  $
  and
  $
    \Vert \mathcal{SS}(\chi')-d\Vert_1=\sum\limits_{i=1}^n\left|\mathcal{SS}(\chi',i)-d_i\right|\le \frac{1}{3}.
  $
\end{conjecture}

\section{An ILP formulation for the inverse power index problem}
\label{section_ilp}

Our first exact approach to solve the inverse power index problem is to utilize an ILP formulation using \texttt{ILOG CPLEX 11} as an ILP solver. Therefore let $n$ be the number of voters and $d=\left(d_1,\dots,d_n\right)$ the desired power distribution with $d_i\ge 0$ and $\sum\limits_{i=1}^nd_i=1$. To model the monotone Boolean function $\chi$ we use binary variables $x_U$ which equal $\chi(U)$ for all subsets $U\subseteq N$. In order to express the values $\mathcal{SS}(\chi,i)$ and $\mathcal{BZ}(\chi,i)$ of the power indices for voter $i$ we additionally introduce binary variables $y_{i,U}$ which equal one iff $U$ is an~$i$-swing, see Equation~(\ref{eq_x_y}) With this we can state the Shapley-Shubik value for voter $i$ as
$$
  \mathcal{SS}(i)=p_i=\frac{1}{n!}\cdot\sum\limits_{U\subseteq N\backslash\{i\}} \Big(|U|!\cdot(n-|U|-1)!\Big)\cdot y_{i,U}
$$
and the number of $i$-swings as
$
  s_i=\sum\limits_{U\subseteq N\backslash\{i\}} y_{i,U}
$
for all $1\le i\le n$. To measure the deviation $\left|p_i-d_i\right|$ for voter~$i$ we introduce variables $\delta_i$ and obtain the following binary ILP formulation for the inverse Shapley-Shubik index problem with $\Vert\cdot\Vert_1$-norm on simple games:
\begin{align}
  \min && \sum_{i=1}^n \delta_i\\
  s.t.{} &&  p_i\le d_i+\delta_i&&\forall 1\le i\le n,\label{eq_dev_1}\\
  &&  p_i\ge d_i-\delta_i&&\forall 1\le i\le n,\label{eq_dev_2}\\
  && p_i=\frac{1}{n!}\cdot \sum\limits_{U\subseteq N\backslash\{i\}} \Big(|U|!\cdot(n-|U|-1)!\Big)\cdot y_{i,U}&&\forall 1\le i\le n,\label{eq_sg_power}\\
  && y_{i,U}=x_{U\cup\{i\}}-x_U&&\forall 1\le i\le n, U\subseteq N\backslash\{i\},\label{eq_x_y}\\
  && x_U\ge x_{U\backslash\{j\}}&& \forall \emptyset\neq U\subseteq N, j\in U,\label{eq_sg1}\\
  && x_\emptyset=0\label{eq_sg2}\\
  && x_N=1\label{eq_sg3}\\
  && x_U\in\{0,1\}&&\forall U\subseteq N,\label{eq_sg4}\\
  && y_{i,U}\in\{0,1\}&&\forall 1\le i\le n,U\subseteq N\backslash\{i\},\\
  && \delta_i\ge 0 &&\forall 1\le i\le n,\\
  && p_i\ge 0 &&\forall 1\le i\le n.
\end{align}

Of course we may directly eliminate all variables $y_{i,U}$ and $p_i$ from this formulation. The equations and inequalities (\ref{eq_sg1}), (\ref{eq_sg2}), (\ref{eq_sg3}), and (\ref{eq_sg4}) enforce that $\chi$ is a simple game. To restrict $\chi$ to complete simple games only small adjustments are necessary: We replace the subsets $U\subseteq N$ by its characteristic vectors $u\in\{0,1\}^n$ and rewrite Inequality~(\ref{eq_sg1}) to 
$x_u\le x_v$ for all $u,v\in\{0,1\}^n$ with $u\preceq v$, or more economically to 
\begin{equation}
  \label{ie_shift}
  x_u\le x_v\quad\quad\forall v\in\{0,1\}^n,u\in \to v :\Vert v\Vert_1\neq 0,
\end{equation}
where $\to v$ denotes the set of elements arising from $v$ by an right-shift of one of the rightmost $1$s in $v$. Examples are given by $\to (1,0,1,0)=\{(1,0,0,1),(0,1,1,0)\}$ and $\to (1,0,1)=\{(1,0,0),(0,1,1)\}$.

To restrict $\chi$ to weighted voting games we additionally have to introduce weights $w_i\ge 0$ and a quota $q>0$. To be able to decide whether $\sum\limits_{i\in U}w_i\ge q$ or $\sum\limits_{i\in U}w_i< q$ in the framework of ILPs we restrict ourself to weights $w_i$ which differ by at least one whenever they are different. (We may simply use integer weights, which could result in harder problems for the ILP solver.) To interlink the $x_U$ with the $w_i$ and $q$ we use\\[-6mm]
\begin{equation}
  q-(1-x_U)\cdot M-\sum\limits_{i\in U} w_i\le 0\quad\quad\forall U\subseteq N\label{iq_big_M_1}
\end{equation}\\[-6mm]
and\\[-6mm]
\begin{equation}
  -x_U\cdot M+\sum\limits_{i\in U} w_i\le q-1\quad\quad\forall U\subseteq N\label{iq_big_M_2},
\end{equation}\\[-4mm]
where $M$ is an suitably large constant fulfilling $M-1\ge\sum\limits_{i=1}^n w_i$. (We may choose $M=4n\left(\frac{n+1}{4}\right)^{(n+1)/2}$, see \cite[Theorem 9.3.2.1]{0243.94014}.) Now let us argue why Inequalities~(\ref{iq_big_M_1}) and (\ref{iq_big_M_2}) are necessary and sufficient. If $x_U=1$ in (\ref{iq_big_M_1}) and $x_U=0$ in (\ref{iq_big_M_2}) then the inequalities are equivalent to $\sum\limits_{i\in U} w_i\ge q$ and $\sum\limits_{i\in U} w_i\le q-1$, respectively. Otherwise (for $x_U\in\{0,1\}$) both sets of inequalities are valid due to $M-1\ge\sum\limits_{i=1}^n w_i$. So we have given ILP formulations for the considered three types of voting methods when the Shapley-Shubik index and the $\Vert\cdot\Vert_1$-norm is used. 

For the Banzhaf index we have to tackle the fact that its value for voter~$i$ is given by the following non-linear equality
$
  \mathcal{BZ}(i)=\frac{s_i}{\sum\limits_{j=1}^n s_j}
$
and we have to evaluate the deviation
$
  \sum\limits_{i=1}^n\left|\mathcal{BZ}(i)-d_i\right|=\left|\frac{s_i}{\sum\limits_{j=1}^n s_j}-d_i\right|
$.

To this end we set $s=\sum\limits_{i=1}^ns_i$ and $\delta_i=\left|s_i-d_i\cdot s\right|$. Instead of minimizing the sum of the $\delta_i$ we introduce
$\sum\limits_{i=1}^n \delta_i\le \alpha\cdot s$,
where $\alpha\in(0,2]$ is a properly chosen constant. With this we can state: If the following binary ILP has a solution then the corresponding simple game~$\chi$ approximates the desired power distribution~$d$ with an error in the $\Vert\cdot\Vert_1$-norm of at most $\alpha$. Otherwise no such approximation is possible. Thus we can minimize the deviation in the $\Vert\cdot\Vert_1$-norm by performing a bisection on $\alpha$. (Since $s$ lies between $n$ and $\left(\left\lfloor\frac{n}{2}\right\rfloor+1\right)\cdot{n\choose \left\lfloor\frac{n}{2}\right\rfloor+1}$, each two unequal Banzhaf vectors differ by at least $\left(\frac{1}{n2^n}\right)^2$, and  we only need $O(n)$ such steps.)

\begin{align}
  && \sum\limits_{i=1}^n s_i=s,&&
\end{align}

\begin{align}
  && \sum\limits_{i=1}^n \delta_i\le \alpha\cdot s,&&\\
  &&  s_i\le d_i\cdot s+\delta_i&&\forall 1\le i\le n,\\
  &&  s_i\ge d_i\cdot s-\delta_i&&\forall 1\le i\le n,\\
  && s_i=\sum\limits_{U\subseteq N\backslash\{i\}} y_{i,U}&&\forall 1\le i\le n,\\
  && y_{i,U}=x_{U\cup\{i\}}-x_U&&\forall 1\le i\le n, U\subseteq N\backslash\{i\},\\
  && x_U\ge x_{U\backslash\{j\}}&& \forall \emptyset\neq U\subseteq N, j\in U,\\
  && x_\emptyset=0\\
  && x_N=1\\
  && x_U\in\{0,1\}&&\forall U\subseteq N,\\
  && y_{i,U}\in\{0,1\}&&\forall 1\le i\le n,U\subseteq N\backslash\{i\},\\
  && \delta_i\ge 0 &&\forall 1\le i\le n,\\
  && s_i\ge 0 &&\forall 1\le i\le n.
\end{align}

Of course we can apply similar adjustments for this ILP formulation as we did for the one modeling the Shapley-Shubik index to restrict the voting method $\chi$, described by the $x_U$, to complete simple games or weighted voting games.

\subsection{Simplification for Conjecture \ref{conjecture_1}}

In order to computationally prove Conjecture~\ref{conjecture_1} for the Shapley-Shubik index and small $n$ we remove some variables and constraints from the previous ILP formulation. 
Since for the power distribution $d=(0.75,0.25,0\dots)$ we have $\sum\limits_{i=1}^n \delta_i=d_1+d_2+(1-p_1-p_2)$ we can delete all variables and constraints for $i\ge 3$:
\begin{align*}
  \min && d_1+d_2+1-p_1-p_2\\
  s.t.{} &&  p_i\le d_i+\delta_i&&\forall 1\le i\le 2,\\
  &&  p_i\ge d_i-\delta_i&&\forall 1\le i\le 2,\\
  && p_i=\frac{1}{n!}\cdot \sum\limits_{U\subseteq N\backslash\{i\}} \Big(|U|!\cdot(n-|U|-1)!\Big)\cdot y_{i,U}&&\forall 1\le i\le 2,\\
  && y_{i,U}=x_{U\cup\{i\}}-x_U&&\forall 1\le i\le 2, U\subseteq N\backslash\{i\},\\
  && x_U\ge x_{U\backslash\{j\}}&& \forall \emptyset\neq U\subseteq N, U\cap\{1,2\}\neq\emptyset, j\in U,\\
  && x_\emptyset=0,\\
  && x_N=1,\\
  && x_U\in\{0,1\}&&\forall U\subseteq N,
\end{align*}  
\begin{align*}
  && y_{i,U}\in\{0,1\}&&\forall 1\le i\le 2,U\subseteq N\backslash\{i\},\\
  && \delta_1,\delta_2\ge 0, &&\\
  && p_1,p_2\ge 0.
\end{align*}

\section{An exhaustive enumeration approach for the inverse power index problem}
\label{section_exhaustive_enumeration}

The conceptionally easiest exact algorithm to solve the inverse power index problem is to loop over the finite set of all voting methods and to pick that one that minimizes the deviation between the desired and the achieved power distribution. Table~\ref{table_numbers} proves this approach to be infeasible for simple games and $n\ge 8$. For complete simple games there exists a very effective parameterization, see \cite{complete_simple_games}. But, to our knowledge, nobody has determined the number of complete simple games for $n\ge 10$ yet. On the one hand the fraction of complete simple games which are weighted voting games tends to zero as the number of voters tends to infinity, but on the other hand no parameterization of weighted voting games is known. Several recursive constructions, mostly using super classes of weighted voting games, are presented in the literature \cite{keijzer2,0841.90134,csg,0205.17805}. Here we will describe the approach of \cite{csg,1163.91339} which exhaustively generates the whole class of complete simple games in the first run and afterwards checks each complete simple game whether it is weighted or not via a linear program, see also \cite{0943.91005,minimum_representation}. One variant is given by
\begin{align}
  &&  \sum_{i=1}^n u_iw_i \ge q &&\forall u\in \widehat{W},\\
  &&  1+\sum_{i=1}^n v_iw_i \le q &&\forall v\in \widehat{L},\\
  && w_i\ge w_{i+1} && \forall 1\le i\le n-1,\\
  && w_n,q\ge 0.
\end{align}
For $n=9$ voters it seems inconceivable to solve $284432730174$ linear programs in order to determine all weighted voting games, nevertheless this is -- more or less -- what was done in \cite{da_tautenhahn}\footnote{Additionally some computational tricks and heuristics were used to reduce the number of linear programs.}. A more effective way is described in \cite{minimum_representation} and uses partial complete simple games.

As mentioned before a complete simple game can be characterized by its set $\widehat{W}$ of shift-minimal winning vectors. Now we want to exhaustively generate complete simple games for a given number of voters using an orderly generation approach, see \cite{0392.05001}, i.~e.{} we start with an empty set $W$ and add shift-minimal winning vectors, which are decreasing with respect to the lexicographical ordering $\le_{\text{lex}}$\footnote{Here one may also read the coalition vectors as integers written in their binary expansion and use the ordinary ordering $\le$ of integers.}, step by step. In the corresponding generation tree for a node $W$ all successors $W'$ fulfill $W\subseteq W'$ and $y<_{\text{lex}}x$ for all $x\in W$, $y\in W'\backslash W$. Now let $L$ and $L'$ be the set of all shift-maximal losing coalitions corresponding to $W$ and $W'$. For a given set $W\neq\emptyset$ with lexicographically smallest element $w$ we define $\tilde{L}:=\Big\{v\in\mathbb{B}^n\mid\exists u\in W:\,v\prec u\Big\}\cup\,\Big\{v\in\mathbb{B}^n\mid\nexists u\in W\cup\left\{x\in\mathbb{B}^n\mid x<_{\text{lex}}w \right\}:\,v\succeq u\Big\}$, where $\mathbb{B}=\{0,1\}$
and condense the shift-maximal vectors of $\tilde{L}$ to a set $\dot{L}\subseteq\tilde{L}$. With this we have $\dot{L}\subseteq L$ and $\dot{L}\subseteq L'$ for all successors of $W$. Thus whenever the previous linear program where $\widehat{W}$ is replaced by $W$ and $\widehat{L}$ is replaced by $\dot{L}$ is infeasible, we can prune the whole search tree below node $W$.

Let us have an example for $n=7$ voters to illustrate the proposed algorithm. For the partial set 
$W=\left\{1 1 1 0 0 0 0, 1 1 0 1 0 1 0, 1 0 0 1 1 1 1 \right\}$ of shift-minimal winning vectors we can determine the 
partial set $$\dot{L}=\left\{1 1 0 1 0 0 0, 1 0 1 1 0 1 0, 1 1 0 0 1 1 0, 1 1 0 1 0 0 1, 0 1 0 1 1 1 1\right\}$$ of shift-maximal
losing vectors. Here $1 1 0 1 0 0 0$ is an element of $\dot{L}$ since we have $1 1 0 1 0 0 0\prec 1 1 1 0 0 0 0\in W$. Since
the linear program corresponding to $W$ and $\dot{L}$ is infeasible we know that none of the complete simple games given by
the set $W\cup X$ of shift-maximal winning vectors is weighted, where 
$$X\subset\left\{0 1 0 1 1 1 1, 0 1 0 1 1 1 0, 0 1 0 1 1 0 1, 0 0 1 1 1 1 1, 0 0 1 1 1 1 0, 0 0 1 1 1 0 1\right\}.$$
So at least $7$~complete games are shown to be non-weighted by solving one linear program.

Using the concept of orderly generation and partial complete games drastically reduces computation times. However the number of weighted voting games for $n=10$ voters remains unknown. So this approach is limited to rather small numbers of voters. 

\subsection{A branch\&bound approach for the inverse power index problem for weighted voting games}
\label{subsec_b_and_b}

Since we will see in Section~\ref{section_computational_results} that the LP relaxation of our ILP models from Section~\ref{section_ilp} is rather weak we propose a branch\&bound algorithm based on the discrete structure of partial complete games.

As in Section~\ref{section_exhaustive_enumeration} we start with an empty set $W$ of the shift-minimal winning vectors. Again we generate complete simple games via orderly generation and prune the subtree of $W$ whenever $W$ can not be extended to a weighted voting game. To further prune the search tree we additionally use the information of $W$ and the partial set $\dot{L}$ of shift-maximal losing vectors to determine some of the $x_U$-values in the ILPs of Section~\ref{section_ilp}. We need a computationally cheap lower bound on the deviation between the values of the given power index for all voting methods $\chi$ \textit{below} node $W$ and the desired power distribution $d$. If the best known solution so far is better than this lower bound, we can prune the search tree. A quite natural candidate for such a lower bound is the LP relaxation of our ILP models from Section~\ref{section_ilp}. To know good solutions right in the beginning of our branch\&bound tree we may utilize any heuristics for the inverse power index problem. In this paper we consider Leech's algorithm.

\section{Computational results}
\label{section_computational_results}

\begin{table}[htp!]
  \begin{center}
    \begin{tabular}{lrlrlr}
      \hline
      state & pop. & state & pop.  state & pop.\\
      \hline
      \hline
      Austria        &  8200 & Germany    & 82500 & Netherlands    & 16300\\
      Belgium        & 10400 & Greece     & 11100 & Poland         & 38200\\
      Bulgaria       &  7800 & Hungary    & 10200 & Portugal       & 10100\\
      Cyprus         &   700 & Ireland    &  4100 & Romania        & 21700\\
      Czech Republic & 10500 & Italy      & 58500 & Slovakia       &  5400\\
      Denmark        &  5400 & Latvia     &  3400 & Slovenia       &  2000\\
      Estonia        &  1300 & Lithuania  &  2300 & Spain          & 43000\\
      Finland        &  5200 & Luxembourg &   500 & Sweden         &  9000\\
      France         & 60000 & Malta      &   400 & United Kingdom & 60600\\
      \hline
    \end{tabular}
    \caption{Population (in thousands) of the $27$ member states of the European Union.}
    \label{table_population}
  \end{center}
\end{table}

\begin{table}[htp!]
  \begin{center}
    \begin{tabular}{c|crr|crr}
      \hline
      & \multicolumn{3}{|c|}{\textbf{simple games}} & \multicolumn{3}{|c}{\textbf{complete simple games}}\\
      \textbf{inst.{}} & $\mathbf{\Vert\cdot\Vert_1}$ & \textbf{nodes} & \textbf{time} & $\mathbf{\Vert\cdot\Vert_1}$ & \textbf{nodes} & \textbf{time} \\
      \hline
      $\mathbf{EU_{1}}$  & $0.00000\cdot 10^{-0}$ &        1 &  0.01s & $0.00000\cdot 10^{-0}$ &      1 &   0.01s \\
      $\mathbf{EU_{2}}$  & $7.69740\cdot 10^{-2}$ &        1 &  0.01s & $7.69740\cdot 10^{-2}$ &      1 &   0.01s \\
      $\mathbf{EU_{3}}$  & $7.13793\cdot 10^{-2}$ &        5 &  0.05s & $7.13793\cdot 10^{-2}$ &      4 &   0.01s \\
      $\mathbf{EU_{4}}$  & $6.30740\cdot 10^{-2}$ &       28 &  0.01s & $6.30740\cdot 10^{-2}$ &      9 &   0.02s \\
      $\mathbf{EU_{5}}$  & $6.90250\cdot 10^{-2}$ &      303 &  0.09s & $6.90250\cdot 10^{-2}$ &     26 &   0.04s \\
      $\mathbf{EU_{6}}$  & $4.18923\cdot 10^{-2}$ &    18265 &  5.91s & $5.40190\cdot 10^{-2}$ &     35 &   0.11s \\
      $\mathbf{EU_{7}}$  & $2.39402\cdot 10^{-2}$ & 71483741 & 15.88h & $3.75078\cdot 10^{-2}$ &     54 &   0.41s \\
      $\mathbf{EU_{8}}$  & & &                                        & $1.78178\cdot 10^{-2}$ &    210 &   3.08s \\
      $\mathbf{EU_{9}}$  & & &                                        & $6.89922\cdot 10^{-3}$ &   1271 &  95.91s \\
      $\mathbf{EU_{10}}$ & & &                                        & $3.65220\cdot 10^{-3}$ &  19520 & 101.21m \\
      $\mathbf{EU_{11}}$ & & &                                        & $1.31318\cdot 10^{-3}$ & 349125 &  55.25h \\
      \hline
    \end{tabular}
    \caption{Results for the EU using the Shapley-Shubik power index.}
    \label{table_EU_SS_1-2}
  \end{center}
\end{table}

As a realistic input instance for the inverse power index problem we use the $27$~countries of the European Union with population numbers $pop_i$ given in Table~\ref{table_population}. Instance $EU_n$ consists of the $n$ countries with the largest population, where the desired power is given by
$$
  d_i=\frac{\sqrt{pop_i}}{\sum\limits_{j=1}^n\sqrt{pop_j}},
$$
i.~e.{} we apply Penrose square root law. In Table~\ref{table_EU_SS_1-2} and Table~\ref{table_EU_SS_3} we state the results for simple games, complete simple games, and
weighted voting games, respectively, using the Shapley-Shubik power index.

The weighted voting games corresponding to Table~\ref{table_EU_SS_3} can be represented by their quota and their weights yielding: $[1;1]$, $[2;1,1]$, $[2;1,1,1]$, $[3;1,1,1,1]$, $[4;1,1,1,1,1]$, $[14;5,5,4,4,3,3]$, $[18;9,8,7,7,6,5,4]$,\\
$[41;13,11,11,10,9,9,7,6]$, $[92;30,27,27,26,23,22,17,16,14]$, $[109;34,29,29,28,25,23,18,16,14,13]$, and 
$[339;96,82,81,80,71,65,51,44,38,36,36]$.

\begin{table}[htp!]
  \begin{center}
    \begin{tabular}{c|crrc}
      \hline
      & \multicolumn{4}{|c}{\textbf{weighted voting games}}\\
      \textbf{inst.{}} & $\mathbf{\Vert\cdot\Vert_1}$ & \textbf{nodes} & \textbf{time} & \textbf{lower bound}\\
      \hline
      $\mathbf{EU_{1}}$  & $0.00000\cdot 10^{-0}$ &      1 &  0.01s & $0.00000\cdot 10^{-0}$ \\
      $\mathbf{EU_{2}}$  & $7.69740\cdot 10^{-2}$ &      3 &  0.01s & $7.69740\cdot 10^{-2}$ \\
      $\mathbf{EU_{3}}$  & $7.13793\cdot 10^{-2}$ &      9 &  0.01s & $7.13793\cdot 10^{-2}$ \\
      $\mathbf{EU_{4}}$  & $6.30740\cdot 10^{-2}$ &      5 &  0.02s & $6.30740\cdot 10^{-2}$ \\
      $\mathbf{EU_{5}}$  & $6.90250\cdot 10^{-2}$ &     23 &  0.04s & $6.86700\cdot 10^{-3}$ \\
      $\mathbf{EU_{6}}$  & $5.40190\cdot 10^{-2}$ &     29 &  0.17s & $2.52257\cdot 10^{-2}$ \\
      $\mathbf{EU_{7}}$  & $3.75078\cdot 10^{-2}$ &     44 &  0.40s & $4.97897\cdot 10^{-3}$ \\
      $\mathbf{EU_{8}}$  & $1.78178\cdot 10^{-2}$ &    165 &  2.86s & $2.20380\cdot 10^{-3}$ \\
      $\mathbf{EU_{9}}$  & $6.89924\cdot 10^{-3}$ &   1449 & 49.65s & $7.19156\cdot 10^{-4}$ \\
      $\mathbf{EU_{10}}$ & $5.35622\cdot 10^{-3}$ &  27517 & 28.68m & $1.12706\cdot 10^{-3}$ \\
      $\mathbf{EU_{11}}$ & $3.62854\cdot 10^{-3}$ & 676774 & 69.16h & $1.14349\cdot 10^{-4}$ \\
      %
      %
      \hline
    \end{tabular}
    \caption{Results for the EU using the Shapley-Shubik power index.}
    \label{table_EU_SS_3}
  \end{center}
\end{table}

For $n\le 5$ the optimal solutions for simple games and complete simple games coincide. For larger $n$ simple games admit better approximations but require a huge amount of cpu time. For $n\le 9$ the optimal solutions for complete simple games and weighted voting games coincide. We would like to remark that for weighted voting games the Inequalities~(\ref{iq_big_M_1}) and (\ref{iq_big_M_2}) would suffice if the $y$-variables are binaries but adding the valid Inequalities~(\ref{ie_shift}) drastically reduces the necessary cpu time.

\begin{table}
 \begin{center}
   \begin{tabular}{rrrrrrrrrr}
     \hline
     $\mathbf{n}$     &\!1\!&\!2\!&\!3\!&\!4\!&\!5               \!&\!6\!              \!&\!7                 \!&\!8\!\\
     \textbf{LP-bound}&\!0\!&\!0\!&\!0\!&\!0\!&\!$1\cdot 10^{-6}$\!&\!$1\cdot 10^{-6}$\!& $1.4\cdot 10^{-6}$\!&\!0\!\\
      $\mathbf{n}$     &\!9               \!&\!10              \!&\!11              \!&\!12              \!&\!13              \!&\!14              \!&\!15\!\\
      \textbf{LP-bound}&\!$9\cdot 10^{-7}$\!&\!$8\cdot 10^{-7}$\!&\!$1\cdot 10^{-7}$\!&\!$3\cdot 10^{-7}$\!&\!$2\cdot 10^{-7}$\!&\!$8\cdot 10^{-7}$\!&\!$6\cdot 10^{-7}$\!\\
     \hline
   \end{tabular}
   \caption{Lower LP-bounds for the instances $EU_{n}$ for simple games, complete simple games, and weighted voting games.}
   \label{table_lower_lp_bounds_ss_fair}
 \end{center}
\end{table}

We would like to remark that our ILP formulations for the inverse power index problem have a large integrality gap so that the lower bound stays zero for a long time during the solution process. Being more precisely we state the lower LP-bounds for the instances $EU_{n}$ for $1\le n\le 15$ players. It turns out that these bounds coincide for simple games, complete simple, and weighted voting games, but there are differences in the necessary computation times. The simple games for $n\ge 13$ took 14~seconds, 1~minute, and 6~minutes, respectively. For complete simple games the computation times for $n\ge 13$ are 2~seconds, 8~seconds, and 4~minutes. Solving the linear programs for weighted voting games with $n\ge 13$ took substantially longer: 4~minutes, 3~minutes and 55~minutes. For the instances $hard_{n}$, see below, we remark that for all considered voting methods and $2\le n\le 16$ voters the lower LP-bound was given by zero and that we have observed a similar behavior of the computation times.

Since the given LP formulation is that weak, in terms of the integrality gap, we might use the general restrictions on the achievable power vectors from Section~\ref{section_a_priori}. Another possibility to obtain a lower bound is to drop the $x$-variables from our ILP formulations and to aggregate the $y_{i,U}$-variables to $z_{i,|U|}$-variables:
\begin{align}
  \min && \sum_{i=1}^n \delta_i\nonumber\\
  s.t.{} &&  d_i-\delta_i \le p_i\le d_i+\delta_i&&\forall 1\le i\le n,\nonumber\\
  && n!\cdot p_i=\sum\limits_{j=0}^{\left\lceil\frac{n-2}{2}\right\rceil} j!\cdot(n-j-1)!\cdot y_{i,j}&&\forall 1\le i\le n,\nonumber\\
  && \sum_{i=1}^n p_i=1,&&\label{eq_power_sums_to_one}\\
  && z_{i,j} \le 2\cdot {n-1 \choose j} && \forall 1\le i\le n,\,\forall 0\le j\le \left\lfloor\frac{n-2}{2}\right\rfloor,\nonumber\\
  && z_{i,\frac{n-1}{2}} \le {n-1 \choose \frac{n-1}{2}} && \forall 1\le i\le n,\, n\equiv 1\pmod 2,\nonumber\\
  && z_{i,j}\in\mathbb{N}&&\forall 1\le i\le n,\,\forall 0\le j\le \left\lceil\frac{n-2}{2}\right\rceil,\nonumber\\
  && \delta_i\ge 0 &&\forall 1\le i\le n.\nonumber
\end{align}
Although the resulting lower bound, see the sixth column in Table~\ref{table_EU_SS_3} \footnote{We may also add inequalities to ensure that all numbers of swings for a given voter have the same parity.}, is valid even for simple games, it turns out to be very useful, at least for small~$n$. If we drop Equation~(\ref{eq_power_sums_to_one}) the ILP decomposes into $n$~simpler ILPs which can be solved more directly, e.~g.{} via enumeration. 

\begin{table}[htp!]
  \begin{center}
    \begin{tabular}{c|crr|crr}
      \hline
      & \multicolumn{3}{|c|}{\textbf{simple games}} & \multicolumn{3}{|c}{\textbf{complete simple games}}\\
      \textbf{inst.{}} & $\mathbf{\Vert\cdot\Vert_1}$ & \textbf{nodes} & \textbf{time} & $\mathbf{\Vert\cdot\Vert_1}$ & \textbf{nodes} & \textbf{time} \\
      \hline
      $\mathbf{hard_{2}}$  & 0.5000 &    1 &   0.01s & 0.5000 &   1 &   0.01s \\
      $\mathbf{hard_{3}}$  & 0.3333 &    1 &   0.07s & 0.3333 &   1 &   0.07s \\
      $\mathbf{hard_{4}}$  & 0.3333 &    6 &   0.05s & 0.3333 &   7 &   0.01s \\
      $\mathbf{hard_{5}}$  & 0.3333 &   14 &   0.02s & 0.3333 &  10 &   0.02s \\
      $\mathbf{hard_{6}}$  & 0.3333 &   28 &   0.09s & 0.3333 &  22 &   0.03s \\
      $\mathbf{hard_{7}}$  & 0.3333 &   31 &   0.11s & 0.3333 &  35 &   0.05s \\
      $\mathbf{hard_{8}}$  & 0.3333 &   33 &   0.19s & 0.3333 &  43 &   0.12s \\
      $\mathbf{hard_{9}}$  & 0.3333 &   43 &   0.73s & 0.3333 &  54 &   0.33s \\
      $\mathbf{hard_{10}}$ & 0.3333 &   69 &   2.30s & 0.3333 & 127 &   0.78s \\
      $\mathbf{hard_{11}}$ & 0.3333 &  183 &  21.31s & 0.3333 & 126 &   2.85s \\
      $\mathbf{hard_{12}}$ & 0.3333 &  177 & 170.56s & 0.3333 & 207 &  18.62s \\
      $\mathbf{hard_{13}}$ & 0.3333 &  151 & 761.24s & 0.3333 & 221 & 141.79s \\
      $\mathbf{hard_{14}}$ & 0.3333 & 2354 &   4.32h & 0.3333 & 194 &  12.52m \\
      $\mathbf{hard_{15}}$ & 0.3333 & 3309 &  23.49h & 0.3333 & 353 &   1.22h \\
      $\mathbf{hard_{16}}$ &        &      &         & 0.3333 & 829 &  53.67h \\
      %
      %
      \hline
    \end{tabular}
    \caption{Results for the desired power distribution $(0.75,0.25,0,\dots)$ using the Shapley-Shubik power index.}
    \label{table_hard_SS_1-2}
  \end{center}
\end{table}

Following Conjecture~\ref{conjecture_1} and \ref{conjecture_2} for $n\ge 2$ we consider the instances $hard_n$ corresponding to the desired power distribution $d=(0.75,0.25,0,\dots)$. In Table~\ref{table_hard_SS_1-2} and Table~\ref{table_hard_SS_3} we list the results for simple games, complete simple games, and weighted voting games, respectively, using the Shapley-Shubik power index.

\begin{table}[htp!]
  \begin{center}
    \begin{tabular}{c|crr}
      \hline
      & \multicolumn{3}{|c}{\textbf{weighted voting games}}\\
      \textbf{inst.{}} & $\mathbf{\Vert\cdot\Vert_1}$ & \textbf{nodes} & \textbf{time}\\
      \hline
      $\mathbf{hard_{2}}$  & 0.5000 &    1 &   0.01s \\
      $\mathbf{hard_{3}}$  & 0.3333 &    3 &   0.02s \\
      $\mathbf{hard_{4}}$  & 0.3333 &    6 &   0.02s \\
      $\mathbf{hard_{5}}$  & 0.3333 &   13 &   0.03s \\
      $\mathbf{hard_{6}}$  & 0.3333 &   19 &   0.05s \\
      $\mathbf{hard_{7}}$  & 0.3333 &   31 &   0.08s \\
      $\mathbf{hard_{8}}$  & 0.3333 &   36 &   0.17s \\
      $\mathbf{hard_{9}}$  & 0.3333 &   60 &   0.43s \\
      $\mathbf{hard_{10}}$ & 0.3333 &  101 &   1.47s \\
      $\mathbf{hard_{11}}$ & 0.3333 &  153 &   4.05s \\
      $\mathbf{hard_{12}}$ & 0.3333 &  208 &  41.09s \\
      $\mathbf{hard_{13}}$ & 0.3333 &  143 & 252.71s \\
      $\mathbf{hard_{14}}$ & 0.3333 &  339 &  51.14m \\
      $\mathbf{hard_{15}}$ & 0.3333 &  579 &   6.73h \\
      $\mathbf{hard_{16}}$ & 0.3333 & 1087 &  36.80h \\
      \hline
    \end{tabular}
    \caption{Results for the desired power distribution $(0.75,0.25,0,\dots)$ using the Shapley-Shubik power index.}
    \label{table_hard_SS_3}
  \end{center}
\end{table}

In Table~\ref{table_hard_BZ_1-2} and Table~\ref{table_hard_BZ_3} we list the results for simple games, complete simple games, and weighted voting games, respectively, using the Banzhaf power index.

\begin{table}[htp!]
  \begin{center}
    \begin{tabular}{c|crr|crr}
      \hline
      & \multicolumn{3}{|c|}{\textbf{simple games}} & \multicolumn{3}{|c}{\textbf{complete simple games}}\\
      \textbf{inst.{}} & $\mathbf{\Vert\cdot\Vert_1}$ & \textbf{nodes} & \textbf{time} & $\mathbf{\Vert\cdot\Vert_1}$ & \textbf{nodes} & \textbf{time} \\
      \hline
      $\mathbf{hard_{2}}$  & $0.5000000$ &      1 &   0.01s & $0.5000000$ &   1 &   0.01s \\
      $\mathbf{hard_{3}}$  & $0.4000000$ &      6 &   0.01s & $0.4000000$ &   5 &   0.07s \\
      $\mathbf{hard_{4}}$  & $0.4000000$ &     11 &   0.01s & $0.4000000$ &   9 &   0.17s \\
      $\mathbf{hard_{5}}$  & $0.3947368$ &     36 &   0.02s & $0.3947368$ &  13 &   0.25s \\
      $\mathbf{hard_{6}}$  & $0.3947368$ &     46 &   0.10s & $0.3947368$ &  21 &   0.26s \\
      $\mathbf{hard_{7}}$  & $0.3797468$ &    177 &   0.27s & $0.3797468$ &  15 &   0.18s \\
      $\mathbf{hard_{8}}$  & $0.3797468$ &   1923 &   7.80s & $0.3797468$ &  55 &   1.15s \\
      $\mathbf{hard_{9}}$  & $0.3793651$ &  11801 &  82.46s & $0.3793651$ &  61 &   2.42s \\
      $\mathbf{hard_{10}}$ & $0.3793651$ & 167868 &  57.60s & $0.3793651$ &  61 &  17.11s \\
      $\mathbf{hard_{11}}$ & $0.3784640$ & 311199 & 140.85m & $0.3784640$ & 110 & 131.67s \\
      $\mathbf{hard_{12}}$ &             &        &         & $0.3784640$ & 119 &  18.07m \\
      $\mathbf{hard_{13}}$ &             &        &         & $0.3784399$ & 119 &  87.85m \\
      $\mathbf{hard_{14}}$ &             &        &         & $0.3784399$ & 191 &   9.48h \\
      $\mathbf{hard_{15}}$ &             &        &         & $0.3783563$ & 217 & 23.78h \\
      $\mathbf{hard_{16}}$ &             &        &         & $0.3783563$ & 248 &  6.84d \\
      \hline
      %
      %
    \end{tabular}
    \caption{Results for the desired power distribution $(0.75,0.25,0,\dots)$ using the Banzhaf power index.}
    \label{table_hard_BZ_1-2}
  \end{center}
\end{table}

\begin{table}[htp!]
  \begin{center}
    \begin{tabular}{c|crrl}
      \hline
      & \multicolumn{4}{|c}{\textbf{weighted voting games}}\\
      \textbf{inst.{}} & $\mathbf{\Vert\cdot\Vert_1}$ & \textbf{nodes} & \textbf{time} & \textbf{representation} \\
      \hline
      $\mathbf{hard_{2}}$  & $0.5000000$ &   1 &   0.01s & [1;1,0]\\
      $\mathbf{hard_{3}}$  & $0.4000000$ &   6 &   0.17s & [3;2,1,1]\\
      $\mathbf{hard_{4}}$  & $0.4000000$ &  26 &   0.16s & [3;2,1,1,0]\\
      $\mathbf{hard_{5}}$  & $0.3947368$ &  18 &   0.17s & [5;4,1,1,1,1]\\
      $\mathbf{hard_{6}}$  & $0.3947368$ &  45 &   0.28s & [5;4,1,1,1,1,0]\\
      $\mathbf{hard_{7}}$  & $0.3797468$ &  38 &   0.62s & [14;11,3,3,3,2,1,1]\\
      $\mathbf{hard_{8}}$  & $0.3797468$ &  55 &   1.41s & [14;11,3,3,3,2,1,1,0]\\
      $\mathbf{hard_{9}}$  & $0.3793651$ &  50 &   4.00s & [33;26,7,7,7,5,2,2,1,1]\\
      $\mathbf{hard_{10}}$ & $0.3793651$ &  73 &  24.20s & [33;26,7,7,7,5,2,2,1,1,0]\\
      $\mathbf{hard_{11}}$ & $0.3784640$ & 138 & 140.95s & [80;63,17,17,17,12,5,5,2,2,1,1]\\
      $\mathbf{hard_{12}}$ & $0.3784640$ & 108 &  26.64m & [80;63,17,17,17,12,5,5,2,2,1,1,0]\\
      $\mathbf{hard_{13}}$ & $0.3784399$ & 111 &  75.33m & [193;152,41,41,41,29,12,12,5,5,2,2,1,1] \\
      $\mathbf{hard_{14}}$ & $0.3784399$ & 153 &  12.30h & [193;152,41,41,41,29,12,12,5,5,2,2,1,1,0]\\
      $\mathbf{hard_{15}}$ & $0.3783563$ & 151 & 1.31d & [466;367,99,99,99,70,29,29,12,12,5,5,2,2,1,1]\\ 
      $\mathbf{hard_{16}}$ & $0.3783563$ & 116 & 2.34d & [466;367,99,99,99,70,29,29,12,12,5,5,2,2,1,1,0]\\ 
      \hline
    \end{tabular}
    \caption{Results for the desired power distribution $(0.75,0.25,0,\dots)$ using the Banzhaf power index.}
    \label{table_hard_BZ_3}
  \end{center}
\end{table}

%
%

\section{Power distributions which are hard to approximate}
\label{sec_hard_to_approximate}

In the previous section we have computationally shown that the desired power distribution $(0.75,0.25,0,\dots)$ is hard to approximate by Shapley-Shubik or Banzhaf power vectors. Since it is reported that the proposed heuristics for the inverse power index problem perform well on practical instances one may conjecture that achievable power vectors are dense in those regions which are of practical importance and that the bad approximation property of $(0.75,0.25,0,\dots)$ is a very singular phenomenon.

To shed some light on this matter we will consider two dimensional projections of the set of all achievable power distributions for small $n$. For $1\le i<j\le n$ and a (normalized) power index $\mathcal{P}$ we define $\mathcal{A}_{i,j}^{\mathcal{P}}=$
$$
  \left\{(x,y)\mid \exists \text{ wvg }\chi\text{ consisting of $n$ ordered voters with }
  \mathcal{P}(i)=x,\mathcal{P}(j)=y\right\}.
$$
to be the set of all achievable power vectors for ordered weighted voting games consisting of $n$ voters and a given power index $\mathcal{P}$. 
By ordered we mean that voter $i$ is as least as powerful as voter $i+1$ for all $1\le i<n$.

\begin{figure}[!htp]
  \begin{center}
    \includegraphics[width=5cm]{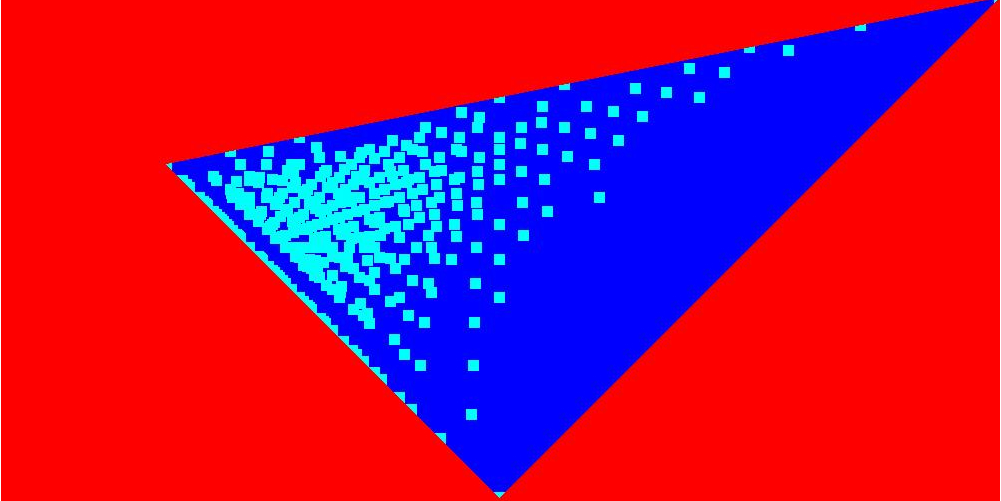}
    \quad\quad
    \includegraphics[width=5cm]{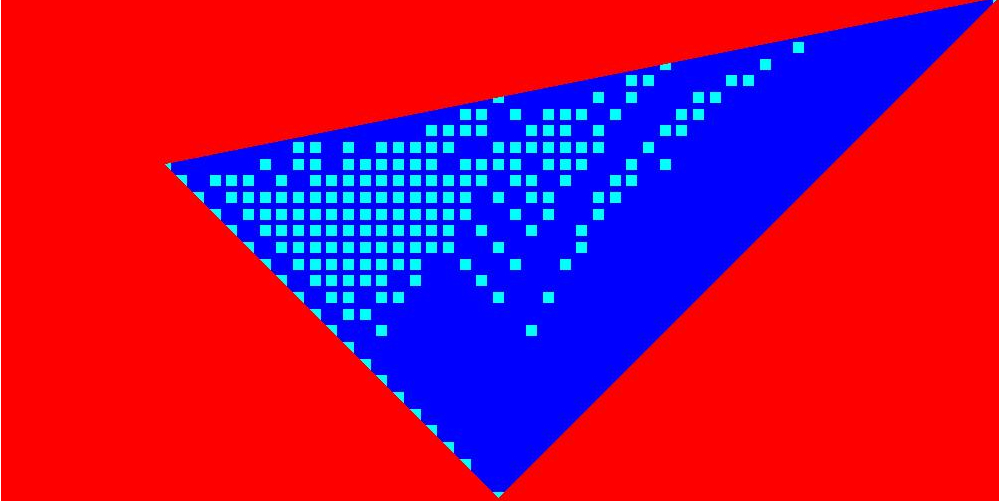}
    
    \medskip
    
    \includegraphics[width=5cm]{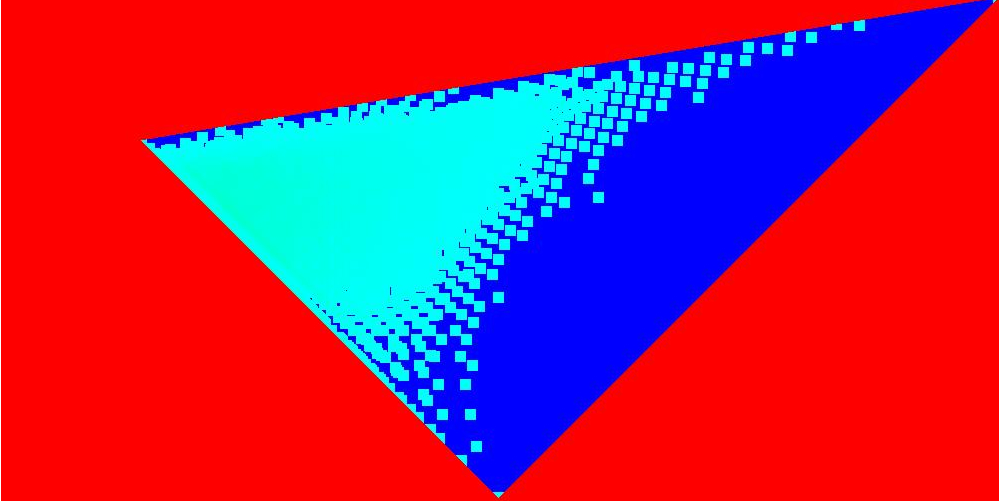}
    \quad\quad
    \includegraphics[width=5cm]{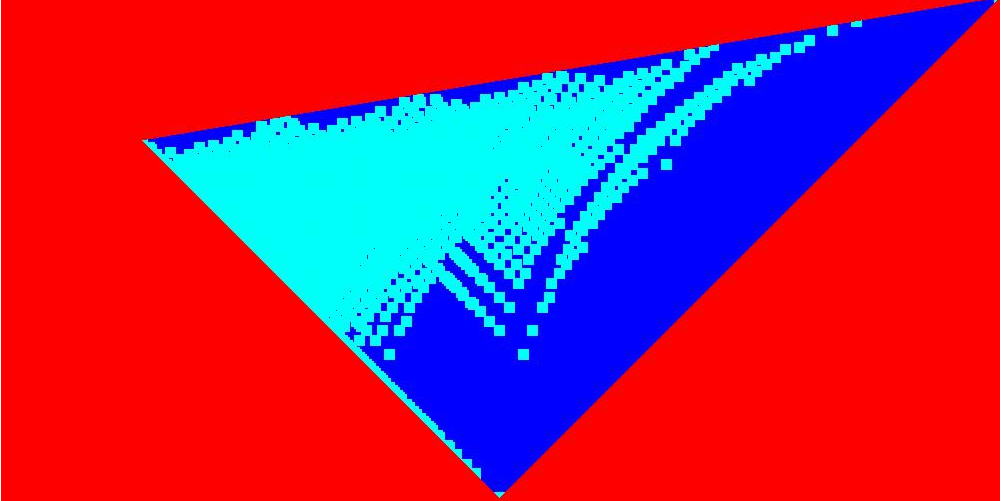}
    
    \medskip
    
    \includegraphics[width=5cm]{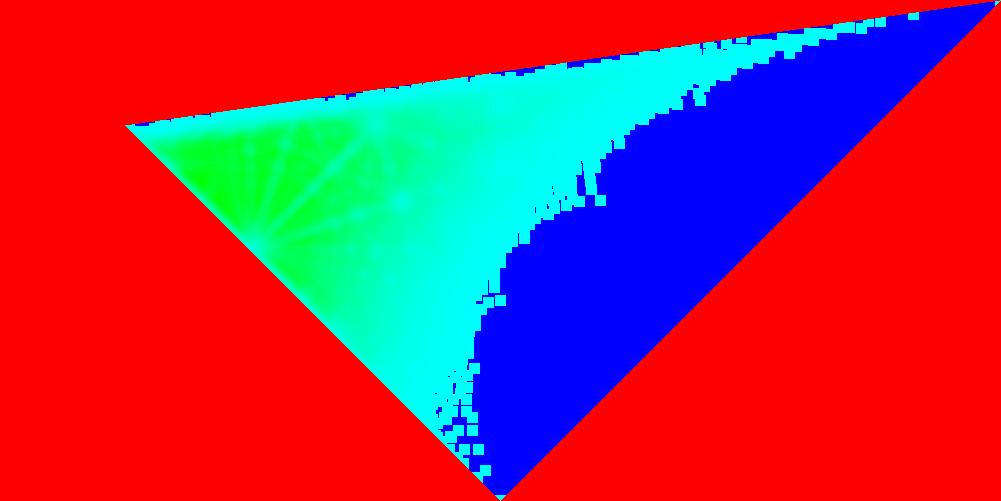}
    \quad\quad
    \includegraphics[width=5cm]{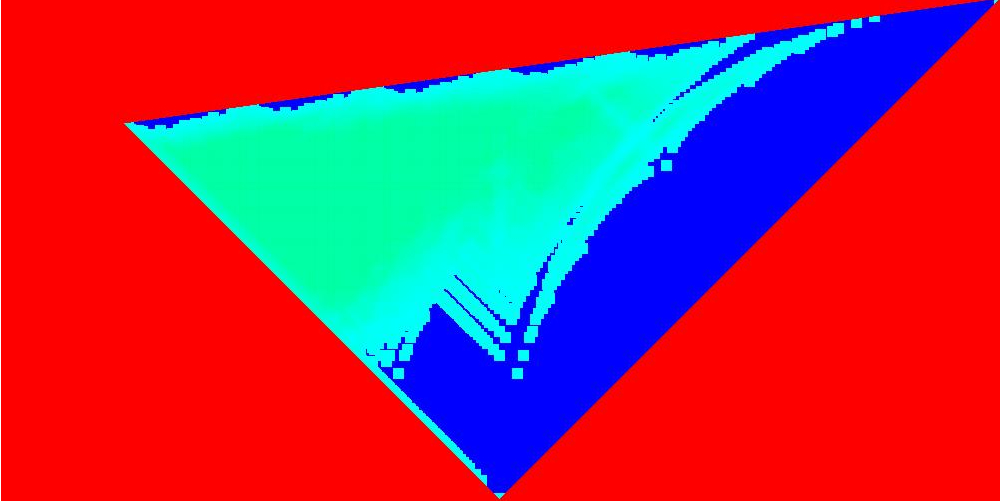}
    
    \medskip
    
    \includegraphics[width=5cm]{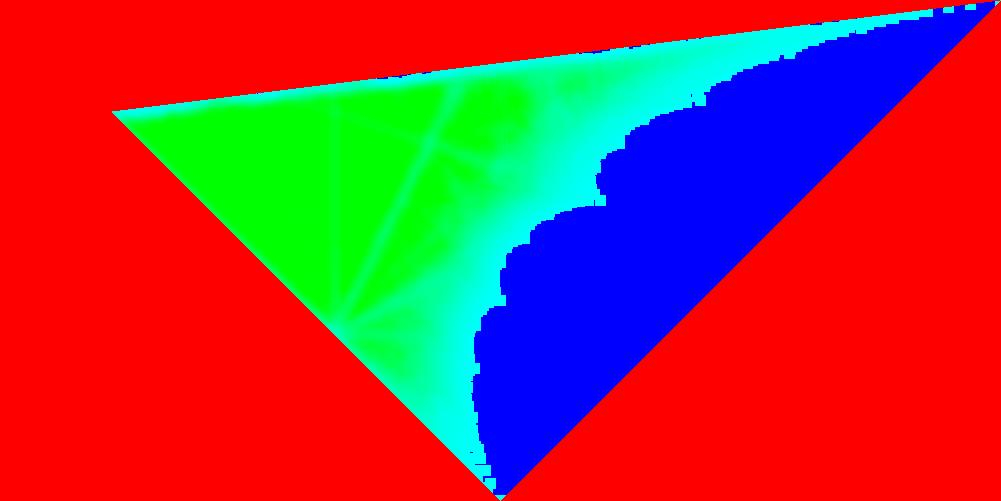}
    \quad\quad
    \includegraphics[width=5cm]{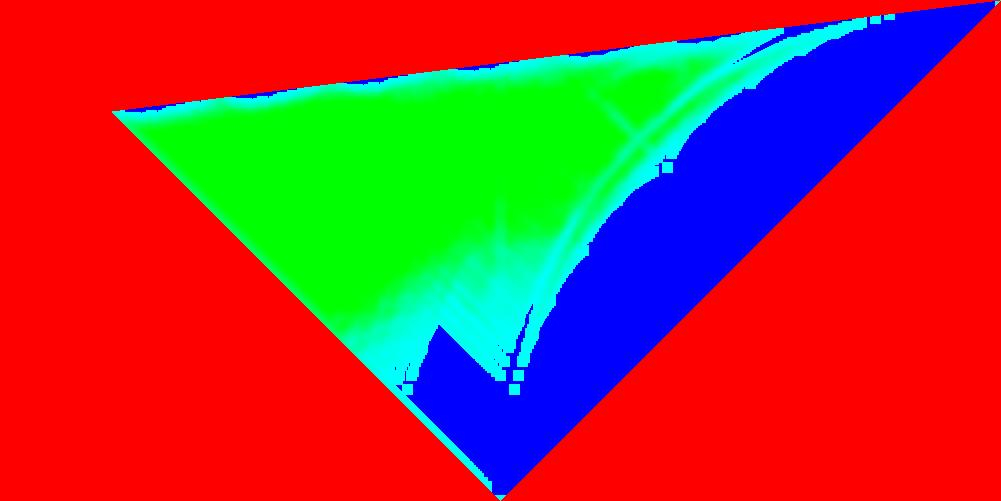}
    \caption{Projections to the first two voters of achievable power distributions for $6\le n\le 9$ voters and the power indices Banzhaf and Shapley Shubik, i.~e.{} the sets
    $\mathcal{A}_{1,2}^{\mathcal{P}}$ for $\mathcal{P}\in\left\{\mathcal{BZ},\mathcal{SS}\right\}$.}
    \label{fig_projections}
  \end{center}
\end{figure}

\begin{lemma}
  \label{lemma_achievable}
  For $(x,y)\in\mathcal{A}_{i,j}^{\mathcal{P}}$ with $1\le i\le n$ we have
  \begin{eqnarray*}
    && 1\ge x\ge y\ge 0,\\   
    && (j-1)\cdot x+(n-j+1)\cdot y\ge 1\text{ for }i=1, \text{ and}\\
    && i\cdot x+(j-i)\cdot y\le 1.
  \end{eqnarray*}
\end{lemma}
\begin{proof}
  Since we assume $\mathcal{P}(1)\ge\mathcal{P}(2)\ge\dots\ge\mathcal{P}(n)$ we have $x\ge y$ and $\mathcal{P}(h)\ge\mathcal{P}(j)$ for all $h\le j$. Inserting this
  into $\sum\limits_{h=1}^n \mathcal{P}(h)=1$ yields the stated inequalities.
\end{proof}

In Figure~\ref{fig_projections} we illustrate these sets for $6\le n\le 9$ and $\mathcal{P}\in\left\{\mathcal{BZ},\mathcal{SS}\right\}$ in the following way. We have $\mathcal{A}_{1,2}^{\mathcal{P}}\subseteq[0,1]\times\left[0,\frac{1}{2}\right]$ and color the pixels of $[0,1]\times\left[0,\frac{1}{2}\right]$ in different colors. For $(x,y)$ is ruled out by Lemma~\ref{lemma_achievable} we use red pixels and otherwise a scale of colors from dark blue to green. Here we use dark blue if $(x,y)$ (actually all points in the near of $(x,y)$) is not contained in $\mathcal{A}_{1,2}^{\mathcal{P}}$. By light blue we denote that some elements of $\mathcal{A}_{1,2}^{\mathcal{P}}$ are in the near of $(x,y)$ and by green we denote the case when there are \textit{many} elements of $\mathcal{A}_{1,2}^{\mathcal{P}}$ are in the near of $(x,y)$. The value for voter~$1$ increases from left to right and the value for voter~$2$ increases from top to bottom. The graphics on the left hand side of Figure~\ref{fig_projections} correspond to the Banzhaf power index and those on the right hand side to the Shapley Shubik power index.

By a look at Figure~\ref{fig_projections} it seems that there might exist a certain subset of $[0,1]^2$ which remains empty (or dark blue) if $n$ tends to infinity and we factor out the intrinsic discrete nature of the problem. But of course proving such a property might be a very challenging task.

\section{Conclusion}
\label{sec_conlusion}

With the inverse power index problem we have presented a very challenging optimization problem of practical importance. Besides exhaustive enumeration we have developed the first exact algorithm to tackle this problem via an ILP formulation. The first computational results are very promising but of course a lot of work using more sophisticated techniques from integer linear (or non-linear) programming has to be done. We give a very precise conjecture on the power distributions with the worst case approximation property and reveal some structure of the set of achievable power distributions by the Banzhaf and the Shapley Shubik power index. In our opinion it would be very challenging to prove these conjectures. Another open problem, which might be not too hard to solve, is to prove an analogue of Alon's and Edelman's result for the Shapley Shubik or other power indices. We leave open the problem to design a practical algorithm for the inverse power index problem, i.~e.{} an algorithm combining fast primal heuristics and techniques for lower bounds in order to solve the problem either exactly or with a certain error guarantee.

\bibliographystyle{plain}
\bibliography{inverse_power_index_problem}

\appendix
\section{Regions which are hard to approximate}
\label{appendix_a}

In Conjecture~\ref{conjecture_1} and Conjecture~\ref{conjecture_2} we state that the desired power distribution $d=(0.75,0.25,0,\dots)$ is the worst possible example with respect to approximation in the $\Vert\cdot\Vert_1$-norm using one of the three presented voting methods and the Shapley-Shubik power index. In this section we want to show that there are further desired power distributions which are equally hard to approximate, at least for small $n$.

We restrict ourselves on complete simple games and weighted voting games, which coincide for $n\le 5$ voters. For $n\le 2$ and the Shapley-Shubik or the Banzhaf power index we only have the following achievable power distributions $s_1^1=(1)$, $s_1^2=(1,0)$, and $s_2^2=(0.5,0.5)$ up to symmetry, i.~e.{} we assume that the power is decreasing monotonely. Let $d=(d_1,\dots,d_n)$ be a theoretically possible desired power distribution, i.~e.{} we have $d_i\ge 0$ and $\sum\limits_{i=1}^nd_i=1$. To avoid symmetric situations we assume $d_i\ge d_{i+1}$ for all $1\le i<n$ in the following.

Since there exists only one such theoretically possible desired power distribution we cannot say something interesting for $n=1$ voters. 

\begin{lemma}
 \label{lemma_classification_2}
 For $n=2$ and the Shapley-Shubik or the Banzhaf power index the regions where $\min\limits_\chi \left\Vert d-\mathcal{P}(\chi)\right\Vert_1=\frac{1}{2}$
 are given by $d=(0.75,0.25)$.
\end{lemma}
\begin{proof}
  If $d_1>0.75$ we have $\left|1-d_1\right|+\left|0-d_2\right|<0.25+0.25=\frac{1}{2}$ and for $\frac{1}{2}\le d_1<0.75$ we have $d_2>0.25$ and
  $\left|0.5-d_1\right|+\left|0.5-d_2\right|<0.25+0.25=\frac{1}{2}$.
\end{proof}

For $n=3$ voters the achievable power distributions for the Shapley-Shubik index up to symmetry are $s_1^3=(1,0,0)$, $s_2^3=\left(\frac{2}{3},\frac{1}{6},\frac{1}{6}\right)$, $s_3^3=\left(\frac{1}{2},\frac{1}{2},0\right)$, and $s_4^3=\left(\frac{1}{3},\frac{1}{3},\frac{1}{3}\right)$.

\begin{lemma}
 For $n=3$ the regions where $\min\limits_\chi \left\Vert d-\mathcal{SS}(\chi)\right\Vert_1=\frac{1}{3}$
 are given by
 \begin{itemize}
   \item $d_1=\frac{5}{6}$, $\frac{1}{12}\le d_2\le \frac{1}{6}$,
   \item $\frac{2}{3}\le d_1\le \frac{5}{6}$, $d_1+d_2=1$ (i.e.~$d_3=0$),
   \item $\frac{1}{2}\le d_1\le \frac{2}{3}$, $d_2=\frac{1}{3}$,
   \item $d_1=\frac{1}{2}$, $\frac{1}{4}\le d_2\le \frac{1}{3}$, and
   \item $\frac{5}{12}\le d_1\le \frac{1}{2}$, $d_1+d_2=\frac{5}{6}$ (i.e.~$d_3=\frac{1}{6}$).
  \end{itemize}
\end{lemma}
\begin{proof}
  \begin{itemize}
   \item[(1)]   For $d_1\ge\frac{5}{6}$ we have $\left\Vert s_1^3-d\right\Vert_1=2-2d_1\le \frac{1}{3}$, where equality is attained iff
                $d_1=\frac{5}{6}$.
   \item[(2.1)] For $\frac{2}{3}\le d_1\le\frac{5}{6}$ and $d_2\le\frac{1}{6}$ we have $\left\Vert s_2^3-d\right\Vert_1=
                2d_1-\frac{4}{3}\le\frac{1}{3}$, where equality is attained iff $d_1=\frac{5}{6}$.
   \item[(2.2)] For $\frac{2}{3}\le d_1\le\frac{5}{6}$ and $d_2\ge\frac{1}{6}$ we have $d_3\le\frac{1}{6}$ and $\left\Vert s_2^3-d\right\Vert_1=
                \frac{1}{3}-2d_3\le\frac{1}{3}$, where equality is attained iff $d_3=0$, i.~e.{} $d_1+d_2=1$.
   \item[(3.1)] For $\frac{1}{2}\le d_1\le\frac{2}{3}$ and $d_3\ge\frac{1}{6}$ we have $\left\Vert s_2^3-d\right\Vert_1=
                \frac{4}{3}-2d_1\le\frac{1}{3}$, where equality is attained iff $d_1=\frac{1}{2}$.
   \item[(3.2)] For $\frac{1}{2}\le d_1\le\frac{2}{3}$, $d_3\le\frac{1}{6}$, and $d_2\le\frac{1}{3}$ we have $d_2\ge\frac{1}{6}$ and $\left\Vert s_2^3-d\right\Vert_1=
                2d_2-\frac{1}{3}\le\frac{1}{3}$, where equality is attained iff $d_2=\frac{1}{3}$.
   \item[(3.3)] For $\frac{1}{2}\le d_1\le\frac{2}{3}$ and $d_2\ge\frac{1}{3}$ we have $d_2\le\frac{1}{2}$ and $\left\Vert s_3^3-d\right\Vert_1=
                1-2d_2\le\frac{1}{3}$, where equality is attained iff $d_2=\frac{1}{3}$.
   \item[(4.1)] For $\frac{5}{12}\le d_1\le\frac{1}{2}$ and $d_3\le\frac{1}{6}$ we have $d_2\le\frac{1}{2}$ and $\left\Vert s_3^3-d\right\Vert_1=
                2d_3\le\frac{1}{3}$, where equality is attained iff $d_3=\frac{1}{6}$.
   \item[(4.2)] For $\frac{5}{12}\le d_1\le\frac{1}{2}$ and $d_2\le\frac{1}{3}$ we have $\left\Vert s_4^3-d\right\Vert_1=
                2d_1-\frac{2}{3}\le\frac{1}{3}$, where equality is attained iff $d_1=\frac{1}{2}$.
   \item[(4.3)] For $\frac{5}{12}\le d_1\le\frac{1}{2}$, $d_3\ge\frac{1}{6}$, and $d_2\ge\frac{1}{3}$ we have $d_3\le\frac{1}{4}<\frac{1}{3}$ and
                $\left\Vert s_4^3-d\right\Vert_1=\frac{2}{3}-2d_3\le\frac{1}{3}$, where equality is attained iff $d_3=\frac{1}{6}$.
   \item[(5.1)] For $\frac{1}{3}\le d_1\le\frac{5}{12}$ and $d_2\ge\frac{1}{3}$ we have $\frac{1}{6}\le d_3\le\frac{1}{3}$ and $\left\Vert s_4^3-d\right\Vert_1=
                \frac{2}{3}-2d_3\le \frac{1}{3}$, where equality is attained iff $d_3=\frac{1}{6}$.
   \item[(5.2)] For $\frac{1}{3}\le d_1\le\frac{5}{12}$ and $d_2\le\frac{1}{3}$ we have $\left\Vert s_4^3-d\right\Vert_1=
                2d_1-\frac{2}{3}<\frac{1}{3}$.
  \end{itemize}
  It remains to check, whether $\left\Vert s_i^3-d\right\Vert_1=\frac{1}{3}$ for the stated subspaces. This is easily verified.
\end{proof}

For the Banzhaf power index and $3$~voters the achievable power distributions up to symmetry are $b_1^3=(1,0,0)$, $b_2^3=\left(\frac{3}{5},\frac{1}{5},\frac{1}{5}\right)$, $b_3^3=\left(\frac{1}{2},\frac{1}{2},0\right)$, and $b_4^3=\left(\frac{1}{3},\frac{1}{3},\frac{1}{3}\right)$.

\begin{lemma}
 For $n=3$ the regions where $\min\limits_\chi \left\Vert d-\mathcal{BZ}(\chi)\right\Vert_1=\frac{2}{5}$
 are given by $d_1=\frac{4}{5}$, $\frac{1}{10}\le d_2\le \frac{1}{5}$ and $d_1=\frac{3}{4}$, $d_2=\frac{1}{4}$.
\end{lemma}
\begin{proof}
  \begin{itemize}
   \item[(1)]    For $d_1\ge \frac{4}{5}$ we have $\left\Vert b_1^3-d\right\Vert_1=2-2d_1\le \frac{2}{5}$, where equality is attained iff
                 $d_1=\frac{4}{5}$.
   \item[(2.1)]  For $\frac{3}{5}\le d_1\le \frac{4}{5}$ and $d_2\ge\frac{1}{5}$ we have $d_3\le\frac{1}{5}$ and 
                 $\left\Vert b_2^3-d\right\Vert_1=\frac{2}{5}-2d_3\le \frac{2}{5}$, where equality
                 is attained iff $d_3=0$. So here we can proceed as in the proof of Lemma~\ref{lemma_classification_2} to deduce
                 $\min\limits_\chi \left\Vert d-\mathcal{BZ}(\chi)\right\Vert_1<\frac{2}{5}$ for $d_1\neq\frac{3}{4}$.
   \item[(2.2)]  For $\frac{3}{5}\le d_1\le \frac{4}{5}$ and $d_2\le\frac{1}{5}$ we have $\left\Vert b_2^3-d\right\Vert_1=2d_1-\frac{6}{5}\le \frac{2}{5}$, where equality
                 is attained iff $d_1=\frac{4}{5}$.
   \item[(3.1)]  For $\frac{1}{2}\le d_1\le \frac{3}{5}$ and $d_2\ge\frac{3}{10}$ we have $d_2\le\frac{1}{2}$ and 
                 $\left\Vert b_3^3-d\right\Vert_1=1-2d_2\le \frac{2}{5}$, where equality
                 is attained iff $d_2=\frac{3}{10}$. For $\frac{1}{2}\le d_1\le \frac{3}{5}$ and $d_2=\frac{3}{10}$ we have $d_3\le\frac{1}{5}$ and
                 $\left\Vert b_2^3-d\right\Vert_1=\frac{1}{5}<\frac{2}{5}$.
   \item[(3.2)]  For $\frac{1}{2}\le d_1\le \frac{3}{5}$, $d_2\le\frac{3}{10}$, and $d_3\le\frac{1}{5}$ we have $d_2\ge\frac{1}{5}$ and
                 $\left\Vert b_2^3-d\right\Vert_1=2d_2-\frac{2}{5}< \frac{2}{5}$.
   \item[(3.3)]  For $\frac{1}{2}\le d_1\le \frac{3}{5}$ and $d_3\ge\frac{1}{5}$ we have $\left\Vert b_2^3-d\right\Vert_1=\frac{6}{5}-2d_1< \frac{2}{5}$.
   \item[(4.1)]  For $d_1\le\frac{1}{2}$ and $d_3\le\frac{1}{5}$ we have $d_2\le\frac{1}{2}$ and $\left\Vert b_3^3-d\right\Vert_1=2d_3\le \frac{2}{5}$, where equality
                 is attained iff $d_3=\frac{1}{5}$.
   \item[(4.2)]  For $\frac{2}{5}\le d_1\le\frac{1}{2}$ and $d_3\ge\frac{1}{5}$ we have $d_2\ge\frac{1}{5}$ and
                 $\left\Vert b_2^3-d\right\Vert_1=\frac{6}{5}-2d_1\le\frac{2}{5}$, where equality
                 is attained iff $d_1=\frac{2}{5}$.
   \item[(5.1)]  For $d_1\le\frac{2}{5}$ and $d_2\ge\frac{1}{3}$ we have $d_1\ge \frac{1}{3}$, $\frac{1}{3}\ge d_3\ge\frac{1}{5}$ and
                 $\left\Vert b_4^3-d\right\Vert_1=\frac{2}{3}-2d_3<\frac{2}{5}$.
   \item[(5.2)]  For $d_1\le\frac{2}{5}$ and $d_2\le\frac{1}{3}$ we have $d_1\ge \frac{1}{3}$ and $\left\Vert b_4^3-d\right\Vert_1=2d_1-\frac{2}{3}<\frac{2}{5}$.              
  \end{itemize} 
  For $d_3=\frac{1}{5}$, see (4.1), we can assume $d_1=d_2=\frac{2}{5}$ due to (4.2) and can apply (5.1). It remains to check, whether
  $\left\Vert b_i^3-d\right\Vert_1=\frac{2}{5}$ for the stated subspaces. This is easily verified.
\end{proof}

For $n=4$ voters the achievable power distributions for the Shapley-Shubik index up to symmetry are $s_1^4=(1,0,0,0)$, $s_2^4=\left(\frac{3}{4},\frac{1}{12},\frac{1}{12},\frac{1}{12}\right)$, $s_3^4=\left(\frac{2}{3},\frac{1}{6},\frac{1}{6},0\right)$, $s_4^4=\left(\frac{7}{12},\frac{1}{4},\frac{1}{12},\frac{1}{12}\right)$, $s_5^4=\left(\frac{1}{2},\frac{1}{2},0,0\right)$, $s_6^4=\left(\frac{1}{2},\frac{1}{6},\frac{1}{6},\frac{1}{6}\right)$, $s_7^4=\left(\frac{5}{12},\frac{5}{12},\frac{1}{12},\frac{1}{12}\right)$, $s_8^4=\left(\frac{5}{12},\frac{1}{4},\frac{1}{4},\frac{1}{12}\right)$, $s_9^4=\left(\frac{1}{3},\frac{1}{3},\frac{1}{3},0\right)$, $s_{10}^4=\left(\frac{1}{3},\frac{1}{3},\frac{1}{6},\frac{1}{6}\right)$, and $s_{11}^4=\left(\frac{1}{4},\frac{1}{4},\frac{1}{4},\frac{1}{4}\right)$.

\medskip

For the Banzhaf power index and $4$~voters the achievable power distributions up to symmetry are $b_1^4=(1,0,0,0)$,
$b_2^4=\left(\frac{7}{10},\frac{1}{10},\frac{1}{10},\frac{1}{10}\right)$, 
$b_3^4=\left(\frac{3}{5},\frac{1}{5},\frac{1}{5},0\right)$, 
$b_4^4=\left(\frac{1}{2},\frac{1}{2},0,0\right)$,
$b_5^4=\left(\frac{1}{2},\frac{3}{10},\frac{1}{10},\frac{1}{10}\right)$,
$b_6^4=\left(\frac{1}{2},\frac{1}{6},\frac{1}{6},\frac{1}{6}\right)$,
$b_7^4=\left(\frac{5}{12},\frac{1}{4},\frac{1}{4},\frac{1}{12}\right)$,
$b_8^4=\left(\frac{2}{5},\frac{1}{5},\frac{1}{5},\frac{1}{5}\right)$,
$b_9^4=\left(\frac{3}{8},\frac{3}{8},\frac{1}{8},\frac{1}{8}\right)$,
$b_{10}^4=\left(\frac{1}{3},\frac{1}{3},\frac{1}{3},0\right)$,
$b_{11}^4=\left(\frac{1}{3},\frac{1}{3},\frac{1}{6},\frac{1}{6}\right)$, and
$b_{12}^4=\left(\frac{1}{4},\frac{1}{4},\frac{1}{4},\frac{1}{4}\right)$.

\begin{lemma}
 For $n=4$ the regions where $\min\limits_\chi \left\Vert d-\mathcal{BZ}(\chi)\right\Vert_1=\frac{2}{5}$
 are given by $d_1=\frac{4}{5}$, $d_2=\frac{1}{5}$ and $d_1=\frac{3}{4}$, $d_2=\frac{1}{4}$.
\end{lemma}
\begin{proof}
  \begin{itemize}
   \item[(1)]    For $d_1\ge \frac{4}{5}$ we have $\left\Vert b_1^4-d\right\Vert_1=2-2d_1\le \frac{2}{5}$, where equality is attained iff
                 $d_1=\frac{4}{5}$. If $d_1=\frac{4}{5}$ and $d_2\le\frac{1}{10}$ then we have $d_4\le d_3\le\frac{1}{10}$ and
                 $\left\Vert b_2^4-d\right\Vert_1=\frac{1}{5}<\frac{2}{5}$.
                 If we have $d_1=\frac{4}{5}$ and $d_2\ge\frac{1}{10}$ then we have $d_2\le\frac{1}{5}$, $d_4\le d_3\le\frac{1}{10}$, 
                 and $\left\Vert b_2^4-d\right\Vert_1=2d_2\le\frac{2}{5}$,
                 where equality is attained iff $d_2=\frac{1}{5}$ and $d_3=d_4=0$.
   \item[(2.1)]  For $\frac{3}{5}\le d_1\le \frac{4}{5}$ and $d_2\ge\frac{1}{5}$ we have $\left\Vert b_3^4-d\right\Vert_1=\frac{2}{5}-2d_3\le \frac{2}{5}$, where equality
                 is attained iff $d_3=d_4=0$. So here we can proceed as in the proof of Lemma~\ref{lemma_classification_2} to deduce
                 $\min\limits_\chi \left\Vert d-\mathcal{BZ}(\chi)\right\Vert_1<\frac{2}{5}$ for $d_1\neq\frac{3}{4}$.
   \item[(2.2)]  For $\frac{3}{5}\le d_1\le \frac{4}{5}$, $d_2\le\frac{1}{5}$, and $d_1+d_4\le\frac{4}{5}$ we have
                 $\left\Vert b_3^4-d\right\Vert_1=2d_1+2d_4-\frac{6}{5}\le \frac{2}{5}$, where equality is attained iff $d_1+d_4=\frac{4}{5}$.
   \item[(2.3)]  For $\frac{3}{5}\le d_1\le \frac{4}{5}$, $d_2\le\frac{1}{5}$, and $d_1+d_4\ge\frac{4}{5}$ we have $d_4\le\frac{1}{10}$, $d_1\ge\frac{7}{10}$, 
                 and $\left\Vert b_2^4-d\right\Vert_1\le \frac{2}{5}$, where equality is only possible for $d_4=0$ so that~(2.2) yields $d_1=\frac{4}{5}$ and we can
                 apply~(1).
   \item[(3.1)]  For $\frac{1}{2}\le d_1\le \frac{3}{5}$ and $d_2\ge\frac{3}{10}$ we have $d_2\le\frac{1}{2}$ and 
                 $\left\Vert b_4^4-d\right\Vert_1=1-2d_2\le \frac{2}{5}$, where equality
                 is attained iff $d_2=\frac{3}{10}$. For $\frac{1}{2}\le d_1\le \frac{3}{5}$ and $d_2=\frac{3}{10}$ we have $d_3+d_4\le\frac{1}{5}$, $d_4\le\frac{1}{10}$, and
                 $\left\Vert b_3^4-d\right\Vert_1=\frac{1}{5}+2d_4\le\frac{2}{5}$ where equality is attained iff $d_3=d_4=\frac{1}{10}$ and $d_1=\frac{1}{2}$. But in the
                 latter case we have $\left\Vert b_5^4-d\right\Vert_1=0$.
   \item[(3.2)]  For $\frac{1}{2}\le d_1\le \frac{3}{5}$, $\frac{1}{5}\le d_2\le\frac{3}{10}$, and $d_3\le\frac{1}{5}$ we have $d_3+d_4= 1-d_1-d_2\ge \frac{7}{10}-d_1$ so that
                 $d_3\ge\frac{7}{20}-\frac{d_1}{2}$ holds. Thus we have $\left\Vert b_3^4-d\right\Vert_1=\frac{8}{5}-2d_1-2d_3\le \frac{9}{10}-d_1\le\frac{2}{5}$, where
                 equality is attained iff $d_1=\frac{1}{2}$, $d_2=\frac{3}{10}$, $d_3=d_4=\frac{1}{10}$, which equals $b_5^4$.
   \item[(3.3)]  For $\frac{1}{2}\le d_1\le \frac{3}{5}$ and $d_3\ge\frac{1}{5}$ we have $d_2\ge\frac{1}{5}$ and $\left\Vert b_3^4-d\right\Vert_1=\frac{6}{5}-2d_1< \frac{2}{5}$.
   \item[(3.4)]  For $\frac{1}{2}\le d_1\le \frac{3}{5}$ and $d_2\le\frac{1}{5}$ we have $d_3\le\frac{1}{5}$ and $\left\Vert b_3^4-d\right\Vert_1=2d_4 <\frac{2}{5}$ since
                 $d_4\le \frac{1-d_1}{3}\le\frac{1}{6}$.
   \item[(4.1)]  For $d_1\le\frac{1}{2}$ and $d_3+d_4\le\frac{1}{5}$ we have $d_2\le\frac{1}{2}$ and $\left\Vert b_4^4-d\right\Vert_1=2d_3+2d_4\le \frac{2}{5}$, where equality
                 is attained iff $d_3+d_4=\frac{1}{5}$. For the latter case see~(4.7).
   \item[(4.2)]  For $\frac{2}{5}\le d_1\le\frac{1}{2}$ and $d_3\ge\frac{1}{5}$ we have $d_2\ge\frac{1}{5}$ and
                 $\left\Vert b_3^4-d\right\Vert_1=\frac{6}{5}-2d_1\le\frac{2}{5}$, where equality is attained iff $d_1=\frac{2}{5}$. For the latter case see~(5.1)-(5.4).
   \item[(4.3)]  For $d_1\le\frac{1}{2}$, $d_2\le\frac{11}{30}$, and $d_3\le\frac{1}{6}$ we have $d_2\ge\frac{1}{6}$, $d_4\le\frac{1}{6}$, and
                 $\left\Vert b_6^4-d\right\Vert_1=2d_2-\frac{1}{3}\le\frac{2}{5}$, where equality is attained iff $d_2=\frac{11}{30}$. In the latter case we consider
                 $x=\left\Vert b_5^4-d\right\Vert_1$. For $d_4\le d_3\le\frac{1}{10}$ we have $x=\frac{2}{15}<\frac{2}{5}$ and for $d_3\ge d_4\ge\frac{1}{10}$ we have
                 $x=2d_3+2d_4-\frac{4}{15}\le\frac{2}{5}$, where equality is attained iff $d_3+d_4=\frac{1}{3}$, i.~e.{} $d_3=d_4=\frac{1}{6}$, and
                 $\left\Vert b_{11}^4-d\right\Vert_1<\frac{2}{5}$. Finally, for $d_3\ge\frac{1}{10}$ and $d_4\le\frac{1}{10}$ we have $x=2d_3-\frac{1}{15}<\frac{2}{5}$.
   \item[(4.4)]  For $\frac{2}{5}\le d_1\le\frac{1}{2}$ and $d_4\ge\frac{1}{6}$ we have $d_2\ge d_3\ge\frac{1}{6}$ and $\left\Vert b_6^4-d\right\Vert_1=1-2d_1<\frac{2}{5}$.
   \item[(4.5)]  For $\frac{2}{5}\le d_1\le\frac{1}{2}$, $d_2\ge\frac{1}{5}$, $d_3\le\frac{1}{5}$, and $d_3+d_4\ge\frac{1}{5}$  we have $d_4\le\frac{1}{5}$ and
                 $\left\Vert b_8^4-d\right\Vert_1=\frac{4}{5}-2d_3-2d_4\le\frac{2}{5}$, where equality is attained iff $d_3+d_4=\frac{1}{5}$. For the latter case see~(4.7).
   \item[(4.6)]  For $\frac{2}{5}\le d_1\le\frac{1}{2}$ and $d_2\le\frac{1}{5}$ we have $d_4\le d_3\le\frac{1}{5}$ and 
                 $\left\Vert b_8^4-d\right\Vert_1=2d_1-\frac{4}{5}<\frac{2}{5}$.
   \item[(4.7)]  For $\frac{2}{5}\le d_1\le\frac{1}{2}$ and $d_3+d_4=\frac{1}{5}$ we have $d_2\ge\frac{3}{10}$, $d_3\ge\frac{1}{10}$, $d_4\le\frac{1}{10}$, 
                 $d_2+d_3\le 1-d_1\le\frac{3}{5}$ and $\left\Vert b_5^4-d\right\Vert_1=2d_2+2d_3-\frac{4}{5}\le\frac{2}{5}$, where equality is attained
                 iff $d_2+d_3=\frac{3}{5}$, i.~e.{} $d_1=\frac{2}{5}$, $d_2=\frac{2}{5}$, $d_3=\frac{1}{5}$, and $d_4=0$. In the latter case we have
                 $\left\Vert b_{10}^4-d\right\Vert_1<\frac{2}{5}$.
   \item[(5.1)]  For $d_1\le\frac{2}{5}$, $d_2\ge\frac{1}{3}$, and $d_3\le\frac{3}{10}$ we have $d_1\ge\frac{1}{3}$, $d_2\le\frac{2}{5}$, $d_4\le\frac{1}{6}$ and 
                 $\left\Vert b_{9}^4-d\right\Vert_1=\left|d_1-\frac{3}{8}\right|+\left|d_2-\frac{3}{8}\right|+\left|d_3-\frac{1}{8}\right|+\left|d_4-\frac{1}{8}\right|
                 \le\frac{1}{24}+\frac{1}{24}+\frac{7}{40}+\frac{1}{8}<\frac{2}{5}$.
   \item[(5.2)]  For $d_1\le\frac{2}{5}$, $d_2\ge\frac{1}{3}$, and $d_3\ge\frac{3}{10}$ we have $d_1\ge\frac{1}{3}$, $d_3\le\frac{1}{3}$, and 
                 $\left\Vert b_{10}^4-d\right\Vert_1=\frac{2}{3}-2d_3<\frac{2}{5}$.
   \item[(5.3)]  For $d_1\le\frac{2}{5}$, $d_2\le\frac{1}{3}$, and $d_4<\frac{2}{15}$ we have $d_1> \frac{13}{45}$. For $d_1\ge\frac{1}{3}$ we have
                 $\left\Vert b_{10}^4-d\right\Vert_1=2d_1+2d_4-\frac{2}{3}<\frac{2}{5}$ and for $d_1<\frac{1}{3}$ we have $\left\Vert b_{10}^4-d\right\Vert_1=2d_4<\frac{2}{5}$.
   \item[(5.4)]  For $\frac{7}{24}\le d_1\le\frac{2}{5}$, $d_2\le\frac{1}{3}$, and $d_4\ge\frac{2}{15}$ we have $d_2\ge\frac{1}{5}$ and consider
                 $x=\left\Vert b_{11}^4-d\right\Vert_1$. If $d_3\ge d_4\ge\frac{1}{6}$ then we have $x=\frac{2}{3}-2d_2<\frac{2}{5}$ for $d_1\ge\frac{1}{3}$ and
                 $x=\frac{4}{3}-2d_1-2d_2<\frac{2}{5}$. If $d_4\le d_3\le\frac{1}{6}$ then we have $x=2d_1-\frac{2}{3}<\frac{2}{5}$ for $d_1\ge\frac{1}{3}$
                 and $x=0$ for $d_1\le\frac{1}{3}$. If $d_3\ge\frac{1}{6}$ and $d_4\le\frac{1}{6}$ then we have $x=2d_1+2d_3-1<\frac{2}{5}$ for $d_1\ge\frac{1}{3}$
                 and $x=2d_3-\frac{1}{3}$ for $d_1\le\frac{1}{3}$.   
   \item[(6)]    For $\frac{1}{4}\le d_1\le\frac{7}{24}$ we have $d_3\le d_2\le\frac{7}{24}$, $d_4=1-d_1-d_2-d_3\ge 1\frac{1}{8}$ and we consider
                 $x=\left\Vert b_{12}^4-d\right\Vert_1$. For $d_4\le d_3\le d_2\le\frac{1}{4}$ we have $x=2d_1-\frac{1}{2}<\frac{2}{5}$, for $d_2\ge\frac{1}{4}$
                 and $d_4\le d_3\le\frac{1}{4}$ we have $x=2d_1+2d_2-1<\frac{2}{5}$, and for $d_2\ge d_3\ge\frac{1}{4}$ and $d_4\le\frac{1}{4}$ we have
                 $x=\frac{1}{2}-2d_4<\frac{2}{5}$.
  \end{itemize}
  Finally we have to show that $\min\limits_\chi \left\Vert d-\mathcal{BZ}(\chi)\right\Vert_1=\frac{2}{5}$ for the stated regions. For $d_1=0.75$, $d_2=0.25$ we have done this
  already by solving an ILP. Now let $d_1=\frac{4}{5}$ and $d_2=\frac{1}{5}$. We can easily check $\left\Vert d-\mathcal{BZ}\left(b_i^4\right)\right\Vert_1>=\frac{2}{5}$ for all
  $1\le i\le 12$ and $\left\Vert d-\mathcal{BZ}\left(b_1^4\right)\right\Vert_1=\frac{2}{5}$.
\end{proof}

We would like to remark that the weighted voting game $\chi=[5;4,1,1,1,1]$ has $\left(\frac{15}{19},\frac{1}{19},\frac{1}{19},\frac{1}{19},\frac{1}{19}\right)$ as its Banzhaf vector and fulfills $\left\Vert (0.8,0.2,0,0,0)-\mathcal{BZ}(\chi)\right\Vert_1<\frac{1}{3}$. So the desired power distribution $(0.8,0.2,0,0,0)$ is not an example for worst possible approximation property. So maybe for $n\ge 5$ the desired power distribution $(0.75,0.25,0,0,0)$ is the unique example with worst possible approximation property for the Banzhaf power index. We remark $\left\Vert (0.75,0.25,0,0,0)-\mathcal{BZ}(\chi)\right\Vert_1=\frac{30}{76}\approx 0.3947368$.

As far as the Shapley Shubik power index is concerned we have no idea if $(0.75,0.25,0,\dots)$ is the unique example with worst possible approximation property for large $n$ or if there are regions which are hard to approximate. In principle we can utilize our ILP formulation from Section~\ref{section_ilp} to check whether a whole region like $d_1=\frac{5}{6}$, $\frac{1}{12}\le d_2\le\frac{1}{6}$ can not be approximated with an error less than $\frac{1}{3}$ by letting the $d_i$ variables with certain range bounds instead of numerical constants. As an example we have tried this for $n=6$ and obtained the weighted voting game $\chi=[5;5,1,1,1,1,1]$ approximating the desired power distribution $d=\left(\frac{5}{6},\frac{1}{12},\frac{1}{12},0,0,0\right)$ with an error of $\left\Vert d-\mathcal{SS}(\chi)\right\Vert_1=\frac{1}{5}<\frac{1}{3}$.

\end{document}